\documentclass[12pt]{article}

\usepackage[margin=1in]{geometry}
\usepackage{amsmath,amsthm,amssymb}
\usepackage{mathtools}
\usepackage{amscd,keyval}

\usepackage{authblk}
\usepackage{fullpage} 

\usepackage[hyperindex,breaklinks]{hyperref}
\hypersetup{colorlinks,linkcolor={blue},citecolor={blue},urlcolor={red}}

\theoremstyle{definition}
\newtheorem{definition}{Definition}[section]

\newtheorem{proposition}[definition]{Proposition}

\theoremstyle{plain}
\newtheorem{theorem}[definition]{Theorem}
\newtheorem{lemma}[definition]{Lemma}

\theoremstyle{remark}
\newtheorem*{remark}{Remark}

\usepackage{microtype}
\begin{document}

\title{A Global Geometric Approach to Parallel Transport of Strings in Gauge Theory}
\author{LI Zimu\thanks{Email: e0267881@u.nus.edu}}
\affil{Department of Mathematics \\
 National University of Singapore, Singapore}
\date{}
\maketitle

\begin{abstract}
     This paper is motivated by recent developments of higher gauge theory. Different from its style of using higher category theory, we try to describe the concept of higher parallel transport within setting of classical principal bundle theory. From this perspective, we obtain a global geometric proof on a generalized 3-dimensional non-abelian Stokes' theorem related to parallel transport on surfaces. It can be naturally extended to four dimension when the underlying crossed module is replaced by crossed 2-module. This 4-dimensional Stokes' theorem yields a global formula for volume holonomy, and also guarantees its gauge invariance. In the process, we also find composition formulas for parallel transport on volumes which have been seen in the case of surfaces.
\end{abstract}


\section*{Introduction}

It has been quite a long time for parallel transport along paths attracting interests from both mathematics and physics. It reveals significant geometric information of connection and curvature on principal bundle, and in gauge theory, it is used to describe the interaction between elementary particles and gauge boson fields. Now, it is also of interest to know could this concept be extended from paths to surfaces in high dimension. Such an extension was thought impossible in non-abelian case during a past period of time. However, problems appearing in physics, like understanding the WZW model \cite{2002RvMaP..14.1281G}, BF theory \cite{1999CMaPh.204..493C} \cite{2008JMP....49c2503G} and flux of magnetic monopoles \cite{2014arXiv1410.6938P}, urge and also promote the establishment of non-abelian higher gauge theory. It is now clear that by replacing common structure group $G$ on principal bundle with the so-called crossed module $(G,H,\tau,\alpha)$, one can deal with non-abelian higher parallel transport (see \cite{2004JMP....45.3949G} and \cite{2004hep.th...12325B} for more details). The latest framework of higher gauge theory was established in \cite{2008arXiv0802.0663S}, \cite{2008arXiv0808.1923S}, \cite{2011arXiv1103.4815N} and \cite{2017arXiv170408542W}. By using tools from category theory, it is natural to categorify a principal bundle as a functor between certain groupoids. Further generalization and abstraction give rise to the notion of principal 2-bundle, and it is now identified as the right place to define higher connection, curvature and parallel transport on surfaces. Furthermore, if crossed module is replaced by crossed 2-module $(G,H,L,\triangleright_H,\triangleright_L,\{-,-\})$, which was first introduced by Conduché in \cite{1984C} as a model for 3-types in homotopy theory, in principle we will touch the case of parallel transport on volumes. Relevant study in this area has already been carried out. On one hand, crossed 2-module draws attention to the research on Gray 3-groupoid and 3-curvature in mathematics \cite{2011arXiv0907.2566FP} \cite{2014JMP....55d3506W}. On the other hand, higher 3-gauge theory also appears in the study of physics \cite{2014LMaPh.104.1147S} \cite{2019arXiv190407566R}. 

\subsection*{Outline and Main Results of this Paper} 

In this article, we wish to investigate a simplified case of global higher parallel transport in terms of classical principal bundle theory. We also proceed in the non-abelian situation, so fake curvature $B$, crossed module $(G,H,\alpha,\tau)$ and their upgraded version: fake 2-curvature $C$, crossed 2-module $(G,H,L,\triangleright_H,\triangleright_L,\{-,-\})$ are necessary. For the convenience of the reader, we briefly review all these concepts in Section \ref{2} and Section \ref{3} respectively. Besides, we formulate another simple but significant notion called standard lift. It is the foundation for our global approach. Basically, in 2-dimensional case a standard lift $\Tilde{\Sigma}$ is a horizontal lift, induced from a connection $A$, of some smooth homotopy $\Sigma$ (or simply a bigon) between smooth paths. Previous study on higher gauge theory requires bigons to be homotopies between paths with fixed endpoints (i.e., $\Sigma(s,t)$ is assigned to a point when $t = 0$ or $1$), since this requirement easily fits in the underlying 2-category structure (see \cite{2014arXiv1410.6938P} and \cite{2017arXiv170408542W} for more details). In this article, bigons with free edges are allowed, so they will be called squares from now on and accordingly we formulate other conditions to make each horizontal lift standard (see Definition \ref{2-4}). Standard lifts are also defined in 3- and 4-dimensional cases for further use (see Definition \ref{2-5} and Definition \ref{3-4}). To be specific, this paper begins with some preliminaries in Section \ref{1}: we develop the notion of parallel transport bundle which serves as the geometric background of standard lifts. This notion is motivated by constructions in \cite{1999CMaPh.204..493C} and \cite{2010RvMaP..22.1033C}, where the entire collection of all horizontal curves lifted from a connection $A$ is discussed. We abstract its essential properties to define parallel transport bundle $\mathcal{P}_{tra}(P)$ associated with any principal bundle $P$ without connection, and it turns out that this formulation is equivalent to connection on principal bundle.   

Then in Section \ref{2}, we rewrite the surface holonomy formula, which is originally a surface-ordered exponential evaluated on bigons contained in local coordinate charts (see \cite{2008arXiv0802.0663S}), by means of standard lifts. This rewritten formula is valid for squares with free edges. Additionally, it is a global formula since squares need not be defined locally, and hence the gauge invariance of this surface holonomy is almost immediate. Parallel transport on squares (Theorem \ref{2-16}) considered in this paper is defined as an analogy of parallel transport along paths. We use the notion of parallel transport bundle and crossed module to establish its geometric background. Based on some consideration from physics, we would like to call it parallel transport of strings: in the context of classical gauge theory, a gauge boson field is treated mathematically as a connection $A$ on a principal bundle $P$ over the spacetime $M$. Thus, when talking about interactions between moving particles and gauge boson fields, parallel transport along paths can come in handy, where paths are thought as the mathematical abstraction of trajectories or worldlines of particles moving in spacetime (see \cite{2017MATHGAUGE} for more historical remarks). Analogously, interactions between strings and gauge boson fields can be understood as parallel transports on squares, where squares are viewed as worldsheets swept out by free strings in spacetime.

One of our main results here is a generalized non-abelian Stokes's theorem presented in Section \ref{2}. In higher gauge theory, it is well known that parallel transports are thin-homotopy invariants. Thin-homotopy means that, for instance in three dimension, the differential of a concerned homotopy $\Theta$ between bigons $\Sigma_0$ and $\Sigma_1$ has at most rank $2$ at each point in its domain $I \times I \times I$ (see \cite{CP94} and \cite{2008arXiv0802.0663S} for more details). Any reparametrization of a bigon induces a thin-homotopy. If two geometric objects (paths or bigons) are thin-homotopic, then their parallel transports must be identical. In  \cite{2008arXiv0802.0663S}, Schreiber and Waldorf also use the thin-homotopy invariance as the base point to encode parallel transports into functors. Since the dependence of surface holonomy on a smooth homotopy of bigons is controlled by the relationship between higher connection (fake curvature $B$) and higher curvature (2-curvature $F_B$), a kind of non-abelian Stokes's theorem in three dimension is necessary. Generally, in the established higher gauge theory, there are lots of ways accessing the proof (see for instance \cite{2010arXiv0710.4310F}, \cite{FP2011}, \cite{2015JGP....95...28S} and \cite{2018arXiv181110060V}). Since we are considering squares swept out by free strings, the involved non-abelian Stokes's formula needs some modification. It is written as (notations are explained in Theorem \ref{2-9}):
\begin{align}
    \Big( \mathcal{P} \exp{\int_s \int_t B_r(\partial_s, \partial_t)} \Big)^{-1} \cdot \mathcal{P} \exp{\int_s \int_t B_0(\partial_s, \partial_t)}  = \mathcal{P} \exp{\int_0^r \int_s Ad_{h_r^{-1}} \int_t F_B(\partial_r, \partial_s, \partial_t) }. \notag
\end{align}
Our assumptions also give us a new and easy proof technically: it is global (i.e., we do not need to restrict ourselves to any local trivialization), and it is only based on a simpler version of non-abelian Stokes's theorem in two dimension. The reasoning process also resembles a mathematical induction, and it can be naturally applied to  $4$-dimensional case in Section \ref{3}. 

When discussing properties of surface holonomy, we obtain a horizontal composition formula: $tra(\Sigma_2 \circ \Sigma_1) = tra(\Sigma_1) \cdot tra(\Sigma_2)$ (see Proposition \ref{2-8}$(c)$). It is different from that in previous work derived from category theory (see for instance \cite{2015JGP....95...28S} and \cite{2018arXiv181110060V}). We realize that the reason is due to some choices under local trivializations, so at the end of Section \ref{2}, we restrict our attention to local computations of surface holonomy. We want to show the correlation between our prospective and some basic aspects of the established higher gauge theory and give our approach more substance through explicit computations.

The last section (Section \ref{3}) is a tentative exploration on parallel transport of cubes. Any cube $\Theta$ concerned here is essentially a smooth homotopy between squares $\Sigma_0$ and $\Sigma_1$ with appropriate conditions (see Definition \ref{3-4}). In the paper, cubes are also allowed to have free edges (cf. \cite{2011arXiv0907.2566FP} and \cite{2014JMP....55d3506W}). As in the case of common curvature $F_A$, higher Bianchi identity on 2-curvature $F_B$, proved by Martins and Picken in \cite{2010arXiv0710.4310F},  tells us that $dF_B + A \wedge_\alpha F_B = 0$. Thus, in order to get access to the notion of 3-curvature, crossed module should be substituted by crossed 2-module. There are comprehensive expositions on these concepts in \cite{2011arXiv0907.2566FP} and \cite{2014JMP....55d3506W}. To make this paper self-contained, we also briefly review them in the beginning of Section \ref{3}. The presentation of this section follows the style of Section \ref{2}: 3- and 4-dimensional standard lifts are defined after the review and used in the proof of $4$-dimensional Stokes' formula (notations are explained in Theorem \ref{3-6}):
\begin{align}
    \bigg( \mathcal{P}  \exp \int_r \int_s \Big( h_{q,r}(s)^{-1} \triangleright'  \big( \int_t C_q \rvert_{rst} - \{\int_t B_q \rvert_{rt}, \int_t B_q \rvert_{st} \} & \big) \Big) \bigg)^{-1} \notag \\
    \cdot \mathcal{P} \exp \int_r \int_s \Big( h_{0,r}(s)^{-1} \triangleright'  \big( \int_t C_0 \rvert_{rst} -  \{\int_t B_0 \rvert_{rt}, & \int_t B_0 \rvert_{st} \}  \big) \Big)  \notag \\
    = \mathcal{P} \exp \int_0^q \int_r Ad_{(l_q(r))^{-1}} \int_s h_{q,r}(s)^{-1} \triangleright' \Big( \int_t & \big( (F_C)_{qrst} - \{ B \wedge B \}_{qrst} \big) \Big). \notag
\end{align}
Based on this formula, we try to define volume holonomy and discuss its basic properties. It should be noted that in \cite{2011arXiv0907.2566FP}, Martins and Picken proposed a local $4$-dimensional Stokes' formula when cubes or 3-paths in their language are pinned at endpoints (i.e., $\Theta(r,s,t)$ is assigned to a point when $t = 0$ or $1$). As in the previous case of squares, this assumption is necessary in categorifying 3-dimensional holonomy by Gray 3-groupoid theory. Arguments in this paper do not depend on that assumption, but for any cube $\Theta$, other restrictions in our setting force $\Theta(r,i,t)$ and $\Theta(r,s,i)$ to be 2-dimensional thin-homotopies when $i=0$ or $1$. This means that, in terms of physics these cubes generally fails to represent trajectories of free branes in spacetime. We expect that further development on higher gauge theory can solve this problem completely.

In short, this paper is organized into three sections:
\begin{enumerate}
    \item[$\bullet$] Section \ref{1} develops the concept of parallel transport bundle $\mathcal{P}_{tra}(P)$. The equivalence between parallel transport bundle and connection is given in Theorem \ref{1-2} and Theorem \ref{1-5}. Theorem \ref{1-4} deals with the non-abelian Stokes's formula in two dimension.
    
    \item[$\bullet$] Section \ref{2} includes the definition of standard lifts (Definition \ref{2-4} and Definition \ref{2-5}), reformulations of surface holonomy (Definition \ref{2-6}) and the generalized non-abelian Stokes's theorem (Theorem \ref{2-9}). The notion of higher parallel transport of strings is given after Theorem \ref{2-16}. At the end of this section, we discuss local calculations on surface holonomy
    
    \item[$\bullet$] Section \ref{3} is divided into two parts: the first part is devoted to the discussion of $4$-dimensional Stokes' formula (Theorem \ref{3-6}). The second part introduces the concept of volume holonomy (Definition \ref{3-7}) and its basic properties (Proposition \ref{3-9}).
\end{enumerate}

\newpage

\subsection*{Basic Notation and Convention}

Let us make the following notation and convention in the beginning of this paper: the capital letter $P$ always denotes a principal $G$ bundle over a base manifold $M$. The space of all smooth paths in $M$ is written as $\mathcal{P}(M)$. A \textbf{smooth family of curves} $\gamma_x$ refers to a smooth map $\phi$ from $I^n \times I$ to its target manifold with $\gamma_x(t)$ defined as $\phi(x,t)$ for each $x \in I^n$, where $I$ the unit interval. When $n = 1$, we prefer the notation $\Sigma$ instead of $\phi$ and call it a square from the curve $\Sigma_0$ to $\Sigma_1$. 

A square $\Sigma$ is said to be a 2-dimensional \textbf{thin-homotopy} if the rank of its differential $d\Sigma$ is at most $1$ at any point in its domain. In this case, paths $\Sigma_0$ and $\Sigma_1$ are said to be thin-homotopic. Given any smooth path $\gamma_1$ with a reparametrization $\gamma_2 \vcentcolon = \gamma \circ f$, where $f$ is an orientation-preserving isomorphic on $I$, it is easy to prove that the straight-line homotopy between $\gamma_1$ and $\gamma_2$ is a thin-homotopy. Thus 2-dimensional thin-homotopy can be thought as a generalization of paths reparametrization and the image of any thin-homotopy $\Sigma$, in the sense of physics, can be viewed as a little perturbation of $\Sigma_0$ (see \cite{CP94}, \cite{2000math......7053M}, \cite{2004hep.th...12325B} and \cite{2008arXiv0802.0663S} for more details).

The ordinary differential equation $dR_{g(t)} f(t) = \Dot{g}(t)$ of smooth maps $g: I \rightarrow G$ and $f: I \rightarrow \mathfrak{g}$ by the right multiplication action $R_g$, is constantly mentioned in the following context. Suppose $g$ is the solution with $g(0) = e_G$, then it is common to denote $g(t)$ by the notation $\mathcal{P} \exp{ \int_0^t f}$. This is the so-called \textbf{path-ordered exponential} of $f$ (see Chapter $5$ of \cite{2017MATHGAUGE} for more details). Note that for a general non-abelian Lie group, this notation may not be thought as the exponential of some integral.

Path-ordered exponential will be used intensively to represent parallel transport in the following context. Suppose $P$ is a principal bundle with a connection one form $A$, or equivalently an Ehresmann horizontal distribution. For any smooth path $\gamma$ in the base manifold $M$, its parallel transport can be obtained by choosing a horizontal lift $\Tilde{\gamma}$: that is, a smooth curve in $P$ lifted from $\gamma$ such that $\Dot{\Tilde{\gamma}}$ is horizontal with respect to $A$. More details will be explained later and can also be found in \cite{2017MATHGAUGE}. Following a similar way, a \textbf{horizontal lift of a square $\Sigma$} is defined as a smooth map $\Tilde{\Sigma}: I \times I \rightarrow P$ such that for each $s \in I$, $\Tilde{\Sigma}_s$ is a horizontal lift of the curve $\Sigma_s$. Just like the case of paths, one can be convinced that for any smooth curve $\eta(s)$ lifted from $\Sigma(s,0)$, there exists a unique horizontal lift $\Tilde{\Sigma}$ of $\Sigma$ with $\Tilde{\Sigma}(s,0) = \eta(s)$ as a sort of initial condition. 


\section{Parallel Transport Bundle}\label{1}

The first section is intended to motivate our investigation of parallel transport. We formulate the concept of parallel transport bundle and discuss its relationship with connection and parallel transport along paths. In Section \ref{2}, this parallel transport bundle is used as the central geometric object in describing parallel transport of string. 

\medskip

\begin{definition}[Parallel Transport Bundle]\label{1-1}
Let $\mathcal{P}_{tra}(P)$ be a collection of smooth curves in $P$ with the natural projection $\Tilde{\pi}: \mathcal{P}_{tra}P \rightarrow \mathcal{P}(M)$.  Elements in $\mathcal{P}_{tra}(P)$ are called curves or lifts. The collection $\mathcal{P}_{tra}(P)$ is said to be a \textbf{parallel transport bundle} associated with $P$ if the following conditions hold: 
\begin{enumerate}
    \item[(a)] Let $\gamma$ be an arbitrary path in $M$ with $\gamma(0) = p$. Then for any point $x_p \in \pi^{-1}(p)$, there exists a lift of $\gamma$ starting from $x_p$.
    
    \item[(b)] Given any lift $\Tilde{\gamma} \in \mathcal{P}_{tra}(P)$, its tangent vector field $\Dot{\Tilde{\gamma}}$ is nowhere vertical in $TP$.
    
    \item[(c)] Suppose both of $\Tilde{\gamma}_1$ and $\Tilde{\gamma}_2$ are lifted from $\gamma$. If $\Tilde{\gamma}_1(0) = \Tilde{\gamma}_2(0) \cdot g$ for some $g \in G$ by the structure group action, then $\Tilde{\gamma}_1(t) = \Tilde{\gamma}_2(t) \cdot g$, for all $t \in I$. (Thus lifts of $\gamma$ is uniquely determined up to their starting points in $P$.)
    
    \item[(d)] Suppose two paths $\gamma_1$ and $\gamma_2$ are smoothly compatible, which means their path composition, denoted by $\gamma_2 \circ \gamma_1$, is still smooth. Let $\Tilde{\gamma}_1$ and $\Tilde{\gamma}_2$ be two lifts of them respectively with $\Tilde{\gamma}_1(1) = \Tilde{\gamma}_2(0)$. Then the composition $\Tilde{\gamma}_2 \circ \Tilde{\gamma}_1$ is lifted from $\gamma_2 \circ \gamma_1$ into $\mathcal{P}_{tra}(P)$ starting from $\Tilde{\gamma}_1(0)$.
    
    \item[(e)] If two paths with the same endpoints in $M$ are thin-homotopic, then any two lifts of these paths starting from the same point in $P$ will terminate at the same point.
    
    \item[(f)] For any local section $s: U \rightarrow P$ of bundle $P$, let $s_\star: \mathcal{P}_U(M) \rightarrow \mathcal{P}_{tra}(P)$ be the map defined on the collection $\mathcal{P}_U(M)$ of all smooth paths starting from $U$ by assigning $s_\star(\gamma)$ to $\Tilde{\gamma}$ with $\Tilde{\gamma}(0) = s(\gamma(0))$. All these local sections $s_\star$ of $\Tilde{\pi}$ are required to be smooth in the following sense: for any smooth family of paths $\phi: I^n \times I \rightarrow M$, the composition $s_\star(\phi(x))(t)$ is smooth from $I^n \times I$ to $P$.
\end{enumerate}
\end{definition}

\begin{remark}
These conditions make $\mathcal{P}_{tra}(P)$ into a principal $G$ bundle, as already indicated by its name. Any local trivialization $(U,\Phi)$ associated with a section $s$ of $P$ induces a local trivialization for $\mathcal{P}_{tra}(P)$ by defining $\Tilde{\gamma} \mapsto (\gamma,g)$, where $g \in G$ is the element such that $s(\gamma(0)) \cdot g = \Tilde{\gamma}(0)$.
\end{remark}

In the study of gauge theory, an equivalence between the so-called parallel transport functor and principal bundle with connection has been obtained by means of category theory (see \cite{2007arXiv0705.0452S}). We try to provide a simplified version in Theorem \ref{1-2} and Theorem \ref{1-5} on the equivalence between parallel transport bundle and connection on a fixed principal bundle $P$, since it is useful for our subsequent discussion. We just sketch the proof in Theorem \ref{1-2} because the details can be found in previous work \cite{1991IJTP...30.1171B} \cite{2000math......7053M} \cite{2004hep.th...12325B} \cite{2007arXiv0705.0452S}. It should be noted that local trivializations are unnecessary in most discussions presented in this paper, and the benefit is shown in the next section. 

\begin{theorem}\label{1-2}
Let $P$ be a principal $G$ bundle. Any parallel transport bundle  $\mathcal{P}_{tra}(P)$ associated with $P$ induces a connection $A$ on $P$. In fact, any curve contained in $\mathcal{P}_{tra}(P)$ is horizontal with respect to the connection $A$. 
\end{theorem}
\begin{proof}
Let us first determine the value $A(v)$ at each point $x_p \in P$ for any tangent vector $v \in T_{x_p} P$: by choosing a smooth curve $\delta(t)$ in $P$ with $\delta(0) = x_p$ and $\Dot{\delta}(0) = v$, we can consider the lift $\Tilde{\gamma}$ of $\gamma = \pi \circ \delta$ in $\mathcal{P}_{tra}(P)$ with $\Tilde{\gamma}(0) = \delta(0)$. There is a unique smooth map $g: I \rightarrow G$ such that $\Tilde{\gamma}(t) = \delta(t) \cdot g(t)$ due to smooth $G$ action on bundle $P$. Then we set $A(v) = -\Dot{g}(0)$. To show that this definition is independent of the choice of $\delta$, thin-homotopy invariance of curves in $\mathcal{P}_{tra}(P)$ is necessary. The importance of this invariance was first noticed by Barrett in \cite{1991IJTP...30.1171B} and used explicitly in \cite{CP94}. We also refer the reader to \cite{2000math......7053M} and \cite{2007arXiv0705.0452S} for the detailed proof.

The process to demonstrate the smoothness and linearity of $A$ is straightforward. Finally, it can be seen, by the construction of $A$, that $g(t)$ is just the solution of $-dR_g A(\Dot{\delta}) = \Dot{g}$ with $g(0) = e_G$ and hence $\Tilde{\gamma} = \delta \cdot g$ is a horizontal lift induced from $A$.
\end{proof}

In the following computations involving differential forms, notation like $A(\partial_s \Tilde{\Sigma}) = \Tilde{\Sigma}^\star A(\partial_s)$ may be abbreviated to $A(\partial_s)$ or even $A_s$ if there is no ambiguity. Next proposition, adapted from \cite{2010RvMaP..22.1033C}, is significant for our discussion, so we quote the proof here further reference. 

\begin{proposition}\label{1-3}
Suppose $P$ is a principal bundle with a connection $A$ or a parallel transport bundle $\mathcal{P}_{tra}(P)$ which will induces a connection (Theorem \ref{1-2}). Let $\Tilde{\Sigma}: I \times I \rightarrow P$ be a smooth map such that $\Tilde{\Sigma}_0$ is a horizontal curve. Then the following statements are equivalent:
\begin{enumerate}
    \item[(a)] Let $A$ be the connection induced by $\mathcal{P}_{tra}(P)$, then $\partial_t A(\partial_s) = - F_A(\partial_s, \partial_t)$, for any $ s,t \in I$.
    \item[(b)] $\Tilde{\Sigma}_s(t) = \Tilde{\Sigma}(s,t)$ is a horizontal lift for any $ s \in I$.
\end{enumerate}
\end{proposition}
\begin{proof}
One direction is immediate: if $(b)$ holds, the structure equation of curvature indicates $(a)$ because $A(\partial_t) = 0$ at any point. 

Conversely, when $(a)$ holds, we have $\partial_s A(\partial_t) + [A(\partial_s), A(\partial_t)] = 0$. For any fixed $t \in I$, let us consider the parallel transport $g_t(s)$ govern by $A$ with respect to the curve $\delta_t(s) = \Tilde{\Sigma}(s,t)$: that is, the solution of $-dR_{g_t} A(\Dot{\delta}_t) = \Dot{g}_t$ with $g_t(0) = e_G$. By taking the matrix form, it is easy to have: 
\begin{align}
\partial_s \big( g_t^{-1} \cdot  A (\partial_t) \cdot g_t \big) & = g_t^{-1} \cdot \big( A (\partial_s) \cdot A(\partial_t) + \partial_s A(\partial_t) - A(\partial_t) \cdot A(\partial_s)  \big) \cdot g_t   \notag \\
& = Ad(g_t^{-1}) \big( \partial_s A(\partial_t) + [A(\partial_s), A(\partial_t)] \big) = 0. 
\end{align}
This indicates that $Ad(g^{-1}) A_{s,t}(\partial_t) = A_{0,t}(\partial_t)$ for $g_t(0) = e_G$. But $A_{0,t}(\partial_t) = 0$ for $\Tilde{\Sigma}_0$ is horizontal. Thus we conclude that $A_{s,t}(\partial_t) = 0$ and $\Tilde{\Sigma}_s$ is a horizontal lift for all $s \in I$.
\end{proof}

The converse part of Theorem \ref{1-2} (i.e., constructing a parallel transport bundle from a given connection $A$), proceeds with taking the collection of all horizontal curves lifted from connection $A$. This structure has been used in \cite{1999CMaPh.204..493C} and \cite{2010RvMaP..22.1033C} to study connection over path space. Confirming the thin-homotopy invariance of these curves (Definition \ref{1-1}$(e)$) seems not so easy. Fortunately, next theorem can help to solve this problem. 

\begin{theorem}\label{1-4}
Suppose $\Sigma: \gamma_0 \Rightarrow \gamma_1$ is a square between paths from point $p$ to $q$. Let $\Tilde{\Sigma}$ be any horizontal lift of $\Sigma$ with $\Tilde{\Sigma}(s,0) = x_p$ for some point $x_p \in \pi^{-1}(p)$ and any $s \in I$ (cf. Definition \ref{2-4}). The difference between parallel transports of $\gamma_0$ and $\gamma_1$ at their terminal point $q$ is hence just difference between $\Tilde{\Sigma}_0(1)$ and $\Tilde{\Sigma}_1(1)$, which is revealed by the path-ordered exponential $g = \mathcal{P} \exp{\int_s \big(\int_t (\Tilde{\Sigma}^\star F_A)(\partial_s, \partial_t) \big) }$: that is, the identity $\Tilde{\Sigma}_0(1) = \Tilde{\Sigma}_1(1) \cdot g$ holds.
\end{theorem}
\begin{proof}
The proof is based on the following simple observation: we consider the curve $\delta(s) = \Tilde{\Sigma}(s,1)$, which is a lift for the constant path $c_q$. The parallel transport $g(s): I \rightarrow G$ of $c_q$, when acting on the right of $\delta(s)$, should yield a constant horizontal lift: that is, $\delta(s) \cdot g(s) = \Tilde{\Sigma}_0(1)$, for all $ s \in I$. Therefore, the value $g(1)$ is the required element. 

Since $g(s)$ is the solution of $- dR_g A(\Dot{\delta}) = \Dot{g}$ with $g(0) = e_G$, its value at $s=1$ can be written as the path-ordered exponential $\mathcal{P} \exp{\int_s -A(\partial_s)}$, where $A(\partial_s)$ is the abbreviation of $A(\Dot{\delta}) = (\Tilde{\Sigma}^\star A)_{s,1}(\partial_s)$. By using the assumption that $\Tilde{\Sigma}$ is a horizontal lift and applying Proposition \ref{1-3}, we obtain that $\mathcal{P} \exp{\int_s -A(\partial_s)} = \mathcal{P} \exp{\int_s \big(\int_t F_A(\partial_s, \partial_t) \big) }$ and this finishes the proof. 

If two paths $\gamma_0$ and $\gamma_1$ happen to be thin-homotopic, then at any point in $\Tilde{\Sigma}$ either $\partial_t$ vanishes or $\partial_s$ is a vertical vector for $\pi:P \rightarrow M$ is a horizontal projection. Therefore, the integral $\int_t \Tilde{\Sigma}^\star F_A(\partial_s,\partial_t) = 0$ and two parallel transports equal to each other. 
\end{proof}

The formula: $\mathcal{P} \exp{\int_s -A(\partial_s)} = \mathcal{P} \exp{\int_s \big(\int_t F_A(\partial_s, \partial_t) \big) }$ is called non-abelian Stokes' theorem in two dimension. There are also other kinds of proof in higher gauge theory (see for instance \cite{2007arXiv0705.0452S} and \cite{2015JGP....95...28S}). 

\begin{theorem}\label{1-5}
Let $P$ be a principal bundle. Any connection $A$ on $P$ induces a parallel transport bundle $\mathcal{P}_{tra}(P)$. Combining with Theorem \ref{1-2}, we obtain a canonical one-to-one correspondence between connections and parallel transport bundles for any fixed bundle $P$.
\end{theorem}

\begin{proof}
Suppose $A$ is a connection on the fixed bundle $P$, let us take the entire collection of horizontal curves to form $\mathcal{P}_{tra}(P)$. The first four conditions in Definition \ref{1-1} are satisfied immediately. $2$-dimensional non-abelian Stokes' theorem ensures the thin-homotopy invariance and the last condition on local sections is guaranteed by the existence of horizontal lifts of squares.
\end{proof}

\medskip

We end this section by reviewing some basic aspects of categorified gauge theory. Any principal bundle $P$ concerned here is assumed to be equipped with a collection of local trivializations $(U_\alpha,\Phi_\alpha)$. In \cite{2007arXiv0705.0452S}, Schreiber and Waldorf have succeeded in categorifying parallel transport along paths as a functor between certain categories. It has also been proved that this functor is equivalent to its descent data which can be thought as a collection of functors $F_\alpha$ defined between thin-path groupoid $\mathcal{P}_1(U_\alpha)$ and structure group $G$ (viewed as a groupoid with only one object). We try to construct $F_\alpha$ from both connection and parallel transport bundle in the next paragraph. This discussion is useful in Theorem \ref{2-9}. 

Let us first recall the standard process for computing parallel transports locally by a given connection $A$: suppose $\gamma$ is a path in $U_\alpha$ with $s_\alpha \circ \gamma$ as a lift, where $s_\alpha$ is the local section associated with $(U_\alpha,\Phi_\alpha)$. Let $A_\alpha$ be the pullback one form of $A$ by $\pi \circ \Phi_\alpha$, and let $g$ be the solution of $-dR_g A_\alpha(\Dot{\gamma}) = \Dot{g}$ with $g(0) = e_G$. When acting on the right of $s_\alpha \circ \gamma$, the new curve $(s_\alpha \circ \gamma) \cdot g$ is a horizontal lift of $\gamma$ (see Chapter 5 of \cite{2017MATHGAUGE}). Since any path, thin-homotopic with $\gamma$, possesses the same parallel transport of $\gamma$ at its terminal point, the functor $F_\alpha$ can be defined by assigning the thin-homotopic class of $\gamma$ to $g(1)$. If we begin with a parallel transport bundle $\mathcal{P}_{tra}(P)$, then the image $F_\alpha(\gamma)$ should be changed to the element $g \in G$ which transits $s_\alpha(\gamma(1))$ to $s_{\alpha \star}(\gamma(1))$ (see Definition \ref{1-1}$(e)$). It is due to Theorem \ref{1-5} that these two ways to construct $F_\alpha$ have the same effect.


\section{Parallel Transport of Strings}\label{2}

In this section, we first summarize the standard definitions of crossed module and fake curvature (see \cite{2004JMP....45.3949G}, \cite{2004hep.th...12325B} and \cite{FP2011} for more details). Then we define the concept of standard lift which is the foundation of our reformulation on surface holonomy. After discussing some basic properties of surface holonomy, we prove the generalized non-abelian Stokes' theorem in three dimension. Then we show how to define parallel transport of strings geometrically by parallel transport bundle. At the end of this section, we try to discuss the correlation between our approach and the established higher gauge theory via computing surface holonomy explicitly through local trivializations.

\bigskip

\begin{definition}\label{2-1}
A \textbf{crossed module} $\{G,H,\alpha,\tau)$ consists of two Lie groups $G$ and $H$ with a Lie group morphism $\tau: H \rightarrow G$ and a Lie group action $\alpha: G \times H \rightarrow H$, such that the following properties hold for any $g \in G$ and $h,h' \in H$:
\begin{enumerate}
    \item[(a)] $\tau(\alpha_g (h)) = Ad_g \tau(h)$.
    \item[(b)] $\alpha_{\tau(h)} h'= Ad_h h'$.
\end{enumerate} 
\end{definition}

Suppose $(G,H,\alpha,\tau)$ is a crossed module. By taking their Lie algebras $\mathfrak{g}, \mathfrak{h}$ and differentials $\tau: \mathfrak{h} \rightarrow \mathfrak{g}$ and $\alpha: \mathfrak{g} \times \mathfrak{h} \rightarrow \mathfrak{h}$ (we use the same notation on the differentials of $\alpha$ and $\tau$ for simplicity), we obtain a \textbf{differential crossed module} with following properties:
\begin{enumerate}
    \item[(a)] For any $X \in \mathfrak{g}$ and $Y \in \mathfrak{h}$, $\tau(\alpha(X, Y)) = [X, \tau(Y)]$.
    \item[(b)] For any $Y,Y' \in \mathfrak{h}$, $\alpha(\tau(Y), Y') = [Y,Y']$.
\end{enumerate}

\begin{definition}\label{2-2}
Suppose $P$ is a principal $G$ bundle with a connection $A$ and a crossed module $(G,H,\alpha,\tau)$, a \textbf{fake curvature} $B$ related to these structure is a two form on $P$ taking value in $\mathfrak{h}$ such that the following conditions hold:
\begin{enumerate}
    \item[(a)] $B$ is horizontal and $\alpha$-equivariant.
    \item[(b)] $\tau(B) = F_A$.
\end{enumerate}
The three form $F_B \vcentcolon = dB + A \wedge_\alpha B$ is called \textbf{2-curvature} of $B$. It has been proved in \cite{2010arXiv0710.4310F} that any 2-curvature is also horizontal and $\alpha$-equivariant.
\end{definition}

It should be noted that all the squares concerned here, without further qualification, are just smooth maps from $I \times I$ to its target manifold. This is not the same with convention in higher gauge theory where squares are pinned at endpoints to fit in the 2-category structure. Squares considered there are called bigons (see \cite{2008arXiv0802.0663S}, \cite{2014arXiv1410.6938P} and \cite{2017arXiv170408542W} for more details). 

\begin{definition}[Cube]\label{2-3}
A \textbf{cube} $\Theta: \Sigma_0 \Rightarrow \Sigma_1$ is a smooth homotopy between squares satisfying the following conditions:
\begin{enumerate}
    \item[(a)]\label{(a)} $\Theta(r,s,i): \Sigma_0(s,i) \Rightarrow \Sigma_1(s,i)$ is a thin homotopy between paths when $i = 0$ or $1$.
    \item[(b)]\label{(b)} $\Theta(r,i,t): \Sigma_0(i,t) \Rightarrow \Sigma_1(i,t)$ is a thin homotopy between paths when $i = 0$ or $1$.
\end{enumerate}
\end{definition}

\begin{remark}
In the context of gauge theory, parallel transport on paths is used to describe interactions between gauge boson fields and particles moving along paths. Similarly, a square $\Sigma: \gamma_0 \Rightarrow \gamma_1$ can be thought as the worldsheet swept out by a string, so any cube $\Theta$ of squares is just a smooth family of worldsheets. Since $\Theta(r,i,t)$ and $\Theta(i,s,t)$ are thin-homotopies, these worldsheets are said to have the same boundary with little perturbations. Based on all these considerations, higher parallel transport on squares with free edges is called parallel transport of strings in this paper.
\end{remark}

\begin{definition}[Standard Lift for squares]\label{2-4}
Let $\Sigma$ be any square, a \textbf{standard lift} of $\Sigma$ is defined as any horizontal lift $\Tilde{\Sigma}_{sd}$ such that  $\Tilde{\Sigma}_{sd}(s,0)$ is a horizontal lift of the path $\Sigma(s,0)$.
\end{definition}

\begin{remark}
It is obvious to see that standard lifts exist: suppose $\Tilde{\Sigma}$ is a horizontal lift of $\Sigma$, we just need to take the parallel transport $g(s)$ of $\Tilde{\Sigma}(s,0)$ and act it on the right of $\Tilde{\Sigma}$. Additionally, it is important to notice that, forced by this definition, two standard lifts $\Tilde{\Sigma}_1$ and $\Tilde{\Sigma}_2$ of the same square $\Sigma$ can differ by only one element $g \in G$ with $\Tilde{\Sigma}_1(s,0) = \Tilde{\Sigma}_2(s,0) \cdot g$. Thus any standard lift is totally determined by its origin point $\Tilde{\Sigma}(0,0)$.
\end{remark}

\begin{definition}[Standard Lift for Cubes]\label{2-5}
The concept of standard lift can be extended to three dimension. We say that $\Tilde{\Theta}$ is a standard lift of a smooth cube $\Theta: \Sigma_0 \Rightarrow \Sigma_1$ if the following conditions hold:
\begin{enumerate}
    \item[(a)] For any fixed $r,s \in I$, $\Tilde{\Theta}_{r,s}(t)$ is a horizontal lift in $t$ direction.
    \item[(b)] For any fixed $r \in I$, $\Tilde{\Theta}_r(s,0)$ is a horizontal lift in $s$ direction.
    \item[(b)] The curve $\Tilde{\Theta}(r,0,0)$ is a horizontal lift in $r$ direction. 
\end{enumerate}
\end{definition}

\begin{remark}
The first two conditions together imply that for each fixed $r \in I$, $\Tilde{\Theta}_r(s,t)$ is a standard lift of the underlying square. Standard lift in high dimension (see Definition \ref{3-4}) is also built step by step. We first deal with the case in $t$ direction as usual. Then we turn to $s$ direction at $t = 0$ and so on. While horizontal curves determined in the preceding steps may be influenced by group action from the following steps. they will remain horizontal for each of them is acted by the group as a whole but not pointise. Similarly to the previous case, standard lifts for cubes are also determined by their origins at $r=s=t=0$. The concept of standard lift is indispensable and essential to our global approach of higher parallel transport and non-abelian Stokes' theorem. 
\end{remark}

Inspired by the work of Schreiber and Waldorf in \cite{2008arXiv0802.0663S}, we reformulate the notion of surface holonomy on common principal bundle. We do not assume any choice of local trivializations and squares are allowed to have free edges. As far as we are aware, it is also the first time to introduce standard lifts in the definition. 

\begin{definition}\label{2-6}
Suppose $P$ is a principal bundle with a connection $A$ and a fake curvature $B$. Let $\Sigma$ be any square in $M$. With respect to a given standard lift $\Tilde{\Sigma}$, the \textbf{surface holonomy} $tra(\Sigma)$ of $\Sigma$ is defined as the path-ordered exponential $\mathcal{P} \exp{\int_s \big(\int_t \Tilde{\Sigma}^\star B (\partial_s, \partial_t) \big) }$.
\end{definition}

Let $\Sigma$ be any square, since there are various ways to lift the point $\Sigma(0,0)$ into its fiber, its surface holonomy may change accordingly. However, it is to be seen in Proposition \ref{2-8}$(a)$ that these differences are totally under control by the group action $\alpha$. Let us introduce the following useful lemma related to the crossed module $(G,H,\alpha,\tau)$ before Proposition \ref{2-8}. 

\begin{lemma}\label{2-7}
Let $h(s): \mathbb{R} \rightarrow H$ be any solution of $\Dot{h} = dR_h f$, where $f(s): \mathbb{R} \rightarrow \mathfrak{h}$ is a smooth map. Then the following statements hold:
\begin{enumerate}
    \item[(a)] The function $\tau(h)$ is a solution of $\Dot{g} = dR_g \tau (f)$.
    \item[(b)] For any $g \in G$, $\alpha_g h$ is a solution of the original equation by substituting $f$ with $\alpha_g f$.
\end{enumerate}
\end{lemma}
\begin{proof}
For any fixed $s \in \mathbb{R}$, suppose $\gamma(t)$ is a smooth curve in $H$ with $\Dot{\gamma}(0) = f(s)$. Then we have the following:
\begin{align}
     \tau(\Dot{h}) & = \tau \big( dR_h f \big) = \frac{d}{dt}\Big\rvert_0 \tau(\gamma \cdot h) = \frac{d}{dt}\Big\rvert_0 \big(\tau(\gamma) \cdot \tau(h) \big) = dR_{\tau(h)} \tau(f), \\ 
    \frac{d}{ds} (\alpha_g h) & = \alpha_g \Dot{h}  = \alpha_g (dR_h f) = \frac{d}{dt}\Big\rvert_0 \alpha_g(\gamma \cdot h) = \frac{d}{dt}\Big\rvert_0 \big( \alpha_g(\gamma) \alpha_g (h) \big) = dR_{\alpha_g(h)} \alpha_g(f). 
\end{align}
These two identities demonstrate the statements.
\end{proof}

\begin{proposition}\label{2-8}
The following properties hold for surface holonomy: (the notation $\bullet$ and $\circ$ are used to denote vertical and horizontal compositions of squares, see \cite{2004hep.th...12325B} for more details.)
\begin{enumerate}
    \item[(a)] For any square $\Sigma$, the surface holonomy $tra(\Sigma)$ is well-defined up to action $\alpha$: that is, if $\Tilde{\Sigma}_1 = \Tilde{\Sigma}_2 \cdot g$ are two standard lifts of $\Sigma$, then $tra_1(\Sigma) = \alpha_{g^{-1}} tra_2(\Sigma)$.
    
    \item[(b)] Suppose two squares $\Sigma_1$ and $\Sigma_2$ are smoothly compatible in the vertical direction. Let $\Tilde{\Sigma}_1$ and $\Tilde{\Sigma}_2$ be two standard lifts such that $\Tilde{\Sigma}_1(1,0) = \Tilde{\Sigma}_2(0,0)$. Then, obviously, $\Tilde{\Sigma}_2 \bullet \Tilde{\Sigma}_1$ is a standard lift of $\Sigma_2 \bullet \Sigma_1$ and $tra(\Sigma_2 \bullet \Sigma_1) = tra(\Sigma_2) \cdot tra(\Sigma_1)$ with respect to these lifts.
    
    \item[(c)] Suppose two squares $\Sigma_1$ and $\Sigma_2$ are smoothly compatible in the horizontal direction. Let $\widetilde{\Sigma_2 \circ \Sigma_1}$ be any standard lift of $\Sigma_2 \circ \Sigma_1$. If we take standard lifts $\Tilde{\Sigma}_1$ and $\Tilde{\Sigma}_2$ of these two squares with $\Tilde{\Sigma}_1(0,0) = \widetilde{\Sigma_2 \circ \Sigma_1}(0,0)$ and $\Tilde{\Sigma}_2(0,0) = \widetilde{\Sigma_2 \circ \Sigma_1}(0,1/2)$ (remember standard lifts are uniquely determined by their origin points), then $tra(\Sigma_2 \circ \Sigma_1) = tra(\Sigma_1) \cdot tra(\Sigma_2)$ with respect to these lifts.
\end{enumerate}
\end{proposition}

\begin{remark}
When written out by local trivializations and restricted to the case of bigons, the last two properties provide the motivation to define parallel transport on bigons as a 2-functor between certain 2-groupoids (see \cite{2008arXiv0802.0663S} and \cite{2014arXiv1410.6938P} for more details). However, it should be careful to notice that our horizontal composition formula is a little bit different from the original one (see for instance \cite{2015JGP....95...28S} and \cite{2018arXiv181110060V}). The dissimilarity is caused by a choice related to local sections, which will be explained in the end of this section.
\end{remark}

\begin{proof}
\begin{enumerate}
\item[(a)] The $\alpha$-equivariance of $B$ indicates that
\begin{align}
 \int_t (\Tilde{\Sigma}_1^\star B) (\partial_s, \partial_t) = \int_t (\Tilde{\Sigma}_2^\star R_g^\star B) (\partial_s, \partial_t)  = \alpha_{a^{-1}} \int_t (\Tilde{\Sigma}_2^\star B) (\partial_s, \partial_t) 
\end{align}
Using Lemma \ref{2-7}$(b)$, we obtain $(a)$ by the uniqueness of solution of differential equation.

\item[(b)] The second property is due to the cocycle condition of the ODE: $\Dot{h} = dR_h f$: that is, if we denote $h(s_0,s)$ as the solution of the ODE with $h(s_0,s_0) = e_H$, then we have $h(s_1,s_2) \cdot h(s_0,s_1) = h(s_0,s_2)$. (see \cite{2008arXiv0802.0663S} and \cite{2015JGP....95...28S} for more details.)

\item[(c)] Let us denote $g(s)$ as the parallel transport for the curve $\widetilde{\Sigma_2 \circ \Sigma_1}(s,1/2)$ and let $h_1(s)$ be the path-ordered exponential $\mathcal{P} \exp{\int_0^r \big( \int_t \Tilde{\Sigma}_1^\star B(\partial_s, \partial_t) \big)}$. Since the curve $\widetilde{\Sigma_2 \circ \Sigma_1}(s,0) = \Tilde{\Sigma}_2(s,0)$ is a horizontal lift, we have $A_{s,0}(\partial_s \Tilde{\Sigma}_1) = 0$. Thus by Proposition \ref{1-3}, $\int_t \Tilde{\Sigma}_1^\star F_A(\partial_s,\partial_t) = - A_{s,1}(\partial_s \Tilde{\Sigma}_1)$. Combining this identity with the fact that $\tau(B) = F_A$ and applying Lemma \ref{2-7}$(a)$, we conclude that $\tau(h_1) = g$ and $\alpha_g = \alpha_{\tau(h_1)} = Ad_{h_1}$, 

Note the relation $\Tilde{\Sigma}_2(s,0) = \widetilde{\Sigma_2 \circ \Sigma_1}(s,1/2) \cdot g(s)$ for $g(s)$ is the parallel transport. Since horizontal curves are uniquely determined when their starting points are fixed, this identity holds even further: for any fixed $s \in I$, the horizontal curve $\Tilde{\Sigma}_2(s,t)$ in $t$ direction equals $\widetilde{\Sigma_2 \circ \Sigma_1}(s,(1+t)/2) \cdot g(s)$. Together with the fact that $B$ is equivariant, we have:
\begin{align}
& \int_t \widetilde{\Sigma_2 \circ \Sigma_1}^\star B(\partial_s,\partial_t) = \int_t \Tilde{\Sigma}_1^\star B (\partial_s,\partial_t) + \int_t \Tilde{\Sigma}_2^\star R_{g(s)^{-1}} B (\partial_s,\partial_t) \notag \\
= & \int_t \Tilde{\Sigma}_1^\star B (\partial_s,\partial_t) + \int_t \Tilde{\Sigma}_2^\star \alpha_{g(s)} B (\partial_s,\partial_t) 
\end{align}
We simply write this identity as $\int_t B_{12} =  \int_t B_1 + \alpha_{g(s)} \int_t B_2$ and denote the path-ordered exponential associated with other two lifts by $h_{12}(s)$ and $h_2(s)$ respectively, then
\begin{align}
    dR_{h_1 \cdot h_2} \big( \int_t B_1 + \alpha_g \int_t B_2 \big) & =  dR_{h_2} \Dot{h}_1 + dR_{h_2} dR_{h_1}  Ad_{h_1} \int_t B_2 \notag \\
    & = dR_{h_2} \Dot{h}_1 +  dR_{h_2} dL_{h_1} \int_t B_2 = \frac{d}{ds} (h_1 \cdot h_2). 
\end{align}
Since $dR_{h_{12}} \int_t B_{12} = \Dot{h}_{12}$, we conclude, by the uniqueness of solution, that $tra(\Sigma_2 \circ \Sigma_1) = tra_(\Sigma_1) \cdot tra(\Sigma_2)$.
\end{enumerate}
\end{proof}

We are now in a position to show the generalized non-abelian Stokes' theorem. As mentioned in the introduction, this kind of theorem deal with the relationship between fake curvature $B$ and its 2-curvature $F_B$, and it reveals the dependence of surface holonomy on a smooth family of squares. Since we are considering a reformulated surface holonomy for squares with free edges, next theorem is a little bit different from its original form in higher gauge theory. Previous approaches to non-abelian Stokes' theorem can be seen in \cite{2010arXiv0710.4310F}, \cite{FP2011}, \cite{2015JGP....95...28S} and \cite{2018arXiv181110060V}. We attempt to make a new proof in the following case.

\begin{theorem}[Generalized Non-abelian Stokes' Theorem]\label{2-9}
Suppose $\Theta: \Sigma_0 \Rightarrow \Sigma_1$ is a cube between squares. With respect to a given standard lift $\Tilde{\Theta}$ of $\Theta$, the difference between $tra(\Sigma_0)$ and $tra(\Sigma_1)$ is revealed by $\mathcal{P} \exp \int_r \big( \int_s Ad_{h_r^{-1}} \int_t (\Tilde{\Theta}^\star F_B) (\partial_r, \partial_s, \partial_t) \big)$ (meaning of $h_r$ is explained in the proof). Moreover, if $\Sigma_0$ and $\Sigma_1$ are thin-homotopic, then they will share the same surface holonomy.
\end{theorem}
\begin{proof}
Let us consider the differential equations $\Dot{h}_r = dR_{h_r} f_r$ indexed by $r \in I$ with $f_r(s) \vcentcolon = \int_t (\Tilde{\Theta}^\star B)_r(\partial s, \partial t)$. The integral is done along the curve $\Tilde{\Theta}_{r,s}(t)$ and we will abbreviate $(\Tilde{\Theta}^\star B)_r(\partial_s, \partial_t)$ to $B_r(\partial_s, \partial_t)$ in the following proof. The solutions $h_r(s)$ of these equations with $h_r(0) = e_H$ form a smooth family of maps. Our strategy is reducing this 3-dimensional problem to two dimension and then applying Theorem \ref{1-4}. 

The first step is defining a one form $a$ on $I \times I$ by setting $a_{r,s}(\partial_s) = -\int_t B(\partial_s, \partial_t)$ and $a_{r,s}(\partial_r) = -\int_t B(\partial_r, \partial_t)$. For each fixed $r \in I$, the map $h_r(s)$ can be seen as the solution of the rephrased ODE $-dR_{h_r}a(\partial s) = \Dot{h}_r$ now. Similarly, we reduce the 2-curvature $F_B = dB + A \wedge_\alpha B$ to a two form $f$ on $I \times I$ with $f_{r,s}(\partial_r, \partial_s) = -\int_t F_B(\partial_r, \partial_s, \partial_t)$. Pairing $f$ with other kinds of vectors is achieved by taking linear combination and applying antisymmetry of differential forms. We claim that $f$ satisfies the structure equation of $a$:
\begin{align}
    - F_a (\partial_r, \partial_s) = & - \partial_r a(\partial_s) +\partial_s a(\partial_r) - [a(\partial_r), a(\partial_s)] \notag \\
    = & \int_t \partial_r B (\partial_s, \partial_t) - \int_t \partial_s B(\partial_r, \partial_t) + [\int_t B(\partial_r, \partial_t), \int_t B(\partial_s, \partial_t)],  
\end{align}
      
\begin{align}
    - f (\partial_r, \partial_s) = & \int_t \big( \partial_r B(\partial_s, \partial_t) - \partial_s B(\partial_r, \partial_t) + \partial_t B(\partial_r, \partial_s) \big) + \notag \\ 
    & \int_t \Big( \alpha \big( A(\partial_r), B(\partial_s, \partial_t) \big) - \alpha \big( A(\partial_s),  B(\partial_r, \partial_t) \big ) + \alpha \big( A(\partial_t),  B(\partial_r, \partial_s) \big ) \Big). 
\end{align}
In the above identities, the last term of $ - f (\partial_r, \partial_s)$ vanishes for $A(\partial_t) = 0$ on these horizontal lifting curves. Since $B$ is a horizontal form, Condition \hyperref[(b)]{(b)} of cube $\Theta$ implies that $B_{r,s,i}(\partial_r, \partial_s) = 0$ when $i=0$ or $1$. Thus the integral $\int_t \partial_t B(\partial_r, \partial_s)$ also vanishes and the first two terms in $-F_a(\partial_r,\partial_s)$ and $-f(\partial_r,\partial_s)$ are identical. 

Now we should examine the equality of rest terms stemming from wedge products. On one hand, let us consider the map $L(t) \vcentcolon = \int_0^t B(\partial_r, \partial_t), \int_0^t B(\partial_s,\partial_t)]$. Its value at $t = 1$ is just the wedge product (Lie bracket) in $-F_a(\partial_r,\partial_s)$. Using the definition of limits and linearity of Lie brackets, we obtain the following identify for its derivative:
\begin{align}
    \Dot{L}(t) = & \lim_{\Delta t \to 0} \Big( [\frac{1}{\Delta t} \int_t^{t+\Delta t} B(\partial_r, \partial_t), \int_0^t B(\partial s, \partial t) ] + [\int_0^t B(\partial_r, \partial_t), \frac{1}{\Delta t} \int_t^{t+\Delta t} B(\partial_s, \partial_t) ] \notag \\
    & + [\frac{1}{\Delta t} \int_t^{t+\Delta t} B(\partial_r, \partial_t), \int_t^{t+\Delta t} B(\partial_s, \partial_t)] \Big) \notag \\
    = & [B(\partial_r,\partial_t), \int_0^t B(\partial_s, \partial_t)] + [\int_0^t B(\partial_r, \partial_t), B(\partial_s,\partial_t)]. 
\end{align}
On the other hand, since $\Tilde{\Theta}(r,0,0)$ is horizontal, $A_{r,0,0}(\partial_r) = 0$. Applying Proposition \ref{1-3} and noting the fact that $\Theta_(r,s,0)$ is a thin-homotopy, we have $A_{r,s,0}(\partial_r) = - \int_s F_A(\partial_r,\partial_s) = 0$. Besides, since $\Tilde{\Theta}_r(s,0)$ is horizontal in $s$ direction, $A_{r,s,0}(\partial_s) = 0$. Combining these facts and using Proposition \ref{1-3} and Lemma \ref{2-7}$(a)$, we have $A_{r,s,t}(\partial_r) = \tau \big( \int_0^t B(\partial_r, \partial_t) \big)$ and $A_{r,s,t}(\partial_s) = \tau \big( \int_0^t B(\partial_s, \partial_t) \big)$. 

Substituting these two identities into the wedge product contained in $-f(\partial_r,\partial_s)$, and applying the second property of differential crossed module (see Definition \ref{2-1}), we have:
\begin{align}
      & \int_t \Big( \alpha \big( A(\partial_r), B(\partial_s, \partial_t) \big) - \alpha \big( A(\partial_s),  B(\partial_r, \partial_t) \big ) \Big) \notag \\
     = &  \int_t \Big( \alpha \big( \tau \big( \int_0^t B(\partial_r, \partial_t) \big), B(\partial_s, \partial_t) \big) - \alpha \big( \tau \big( \int_0^t B(\partial_s, \partial_t) \big),  B(\partial_r, \partial_t) \big ) \Big) \notag \\
     = & \int_t \big( [\int_0^t B(\partial_r, \partial_t), B(\partial_s,\partial_t)] - [\int_0^t B(\partial_s, \partial_t), B(\partial_r,\partial_t)] \big). 
\end{align}
It is just the integral of $\Dot{L}(t)$ to $t=1$, and hence $f = F_a$. 

We are now at the last step. Let us consider the trivial $H$ bundle $(I \times I) \times H$. One-to-one correspondence between local gauge fields and connection (see Chapter $5$ of \cite{2017MATHGAUGE}) indicates that there is a pushforward connection $\Tilde{a}$ with curvature $F_{\Tilde{a}}$ on the trivial bundle induced from $a$ and $f$.

We view the identity map $id$ on $I \times I$ as a square. Similarly to our discussion at the end of Section \ref{1}, the function $h_r$ is just the parallel transport for path $id_r(s) = (r,s)$ with respect to the connection $\Tilde{a}$, and thus $\Tilde{id}(r,s) \vcentcolon = (id_r(s), h_r(s))$ is a horizontal lift for $id$ in the trivial bundle. According to the correspondence between local gauge fields and connection, we have:
\begin{align}
    & \Tilde{a}(\partial_r \Tilde{id}) = Ad_{h_r^{-1}} a \big( \pi_\star (\partial_r \Tilde{id}) \big) + \omega_H (\partial_r h_r) = Ad_{h_r^{-1}} a(\partial_r) + L_{h_r^{-1}} (\partial_r h_r), \\ 
    & F_{\Tilde{a}}(\partial_r \Tilde{id}, \partial_s \Tilde{id}) = Ad_{h_r^{-1}} f \big( \pi_\star (\partial_r \Tilde{id}, \partial_s \Tilde{id}) \big) = Ad_{h_r^{-1}} f (\partial_r, \partial_s),  
\end{align}
where $\omega_H$ is the Maurer-Cartan form on $H$. Condition \hyperref[(a)]{(a)} of cube $\Theta$ forces $a_{r,i}(\partial_r) = \int_t B_{r,i,t}(\partial_r,\partial_t) = 0$ when $i=0$ or $1$. Since $h_r(0) = e_H$, we get the following result by applying Proposition \ref{1-3} to the above two identities:
\begin{align}
    L_{h_r^{-1}(1)} \partial_r h_r(1) = -\int_s Ad_{h_r^{-1}} f(\partial_r, \partial_t) = \int_s  Ad_{h_r^{-1}} \int_t F_B(\partial_r, \partial_s, \partial_t).\label{(13)} 
\end{align}
After rewriting the above equation to its standard form: $\Dot{h} = dR_h f$ and taking path-ordered exponential, the proof is completed with the following non-abelian Stokes' formula:
\begin{align}
    \Big( \mathcal{P} \exp{\int_s \int_t B_r(\partial_s, \partial_t)} \Big)^{-1} \cdot \mathcal{P} \exp{\int_s \int_t B_0(\partial_s, \partial_t)}  = \mathcal{P} \exp{\int_0^r \int_s Ad_{h_r^{-1}} \int_t F_B(\partial_r, \partial_s, \partial_t) }. 
\end{align}
Since $h_r(1) = tra_(\Sigma_r)$, left hand side of this formula stands for the difference between surface holonomies. If $\Theta$ is a thin-homotopy, then at any point in the domain of $\Tilde{\Theta}$ either $\partial_t \Tilde{\Theta}$ vanishes or one of $\partial_r \Tilde{\Theta}$ and $\partial_s \Tilde{\Theta}$ is vertical for $\pi:P \rightarrow M$ is a horizontal projection. Therefore $F_B(\partial_r, \partial_s, \partial_t) = 0$ and $tra_(\Sigma_0) = tra_(\Sigma_1)$.
\end{proof}

By Definition \ref{2-2} and the second property of differential crossed module, values of $F_B$ are all found in the center of $\mathfrak{h}$ and so is the integral $\int_t F_B(\partial_r, \partial_s, \partial_t)$. Thus, if $H$ is a matrix group or compact, the above non-abelian Stokes's formula becomes
\begin{align}
    \Big( \mathcal{P} \exp{\int_s \int_t B_r(\partial_s, \partial_t)} \Big)^{-1} \cdot \mathcal{P} \exp{\int_s \int_t B_0(\partial_s, \partial_t)}  = \mathcal{P} \exp{\int_0^r \int_s \int_t F_B(\partial_r, \partial_s, \partial_t) }. \tag{$14^\star$} 
\end{align}
The thin-homotopy invariance, if spoken in the language of physics, says that any perturbation on the worldsheet of a string in the spacetime $M$ will not influence its interaction with the gauge boson field. (See \cite{2010RvMaP..22.1033C}, \cite{2011GReGr..43.2335B} and \cite{2017MATHGAUGE} for more historical remarks.)


\bigskip

Remember in the case of paths, parallel transport is described by horizontal lifting curves in principal bundle. In the following context, we formulate a similar configuration for parallel transport on squares by means of parallel transport bundle $\mathcal{P}_{tra}(P)$.

\begin{lemma}\label{2-10}
Suppose $(G,H,\alpha,\tau)$ is a crossed module, then the product $G \times H$ is still a Lie group with group multiplication defined as follows: for any $(g_1,h_1),(g_2,h_2) \in G \times H$, $(g_1,h_1) \cdot (g_2,h_2) \vcentcolon = (g_1 \cdot g_2, \ (\alpha_{g_2^{-1}} h_1) \cdot h_2)$. When equipped with this group structure, the produce will be denoted by $G \times_{\alpha} H$.
\end{lemma}

The proof is easy. For instance, identity in the product is just $(e_G,e_H)$ and the inverse of any $(g,h)$ is $(g^{-1}, \alpha_g h^{-1})$. This kind of semiproduct is used to show that any crossed module is equivalent to a 2-groupoid in 2-category theory (see \cite{2004hep.th...12325B} for instance), but it should be noted that the group multiplication concerned here is adjusted to fit in the following definition and different from one in the previous work.

\begin{definition}[The bundle $\mathcal{P}_{tra}(P) \times H$]\label{2-11}
Suppose $(G,H,\alpha,\tau)$ is a crossed module and suppose $P$ is a principal $G$-bundle with a connection $A$. Product space $\mathcal{P}_{tra}(P) \times H$ with natural projection $\Tilde{\pi}(\Tilde{\gamma}, h) \vcentcolon = \gamma$ is a $G \times_{\alpha} H$ bundle over the path space $\mathcal{P}(M)$. The right  $G \times_{\alpha} H$-action on $\mathcal{P}_{tra}(P) \times H$ is given as: $(\Tilde{\gamma}, h') \cdot (g,h) \vcentcolon = (\Tilde{\gamma} \cdot g, \ (\alpha_{g^{-1}} h') \cdot h)$. 
\end{definition}

This action is obviously free and transitive on each fiber of $\mathcal{P}_{tra}(P) \times H$. Like the case of principal bundle, we will formulate the notion of vertical and horizontal spaces of $\mathcal{P}_{tra}(P) \times H$. As the first step, it is necessary to define tangent spaces for parallel transport bundle. We adapt the following definition from previous work in \cite{1999CMaPh.204..493C} and \cite{2010RvMaP..22.1033C}.

\begin{definition}\label{2-12}
The tangent space $T_{\Tilde{\gamma}} \mathcal{P}_{tra}(P)$ at any point $\Tilde{\gamma} \in \mathcal{P}_{tra}(P)$ is defined as the collection of all tangent vector fields $\Tilde{v}$ along curve $\Tilde{\gamma}$ such that the equation $\partial_t A(\Tilde{v}(t)) = - F_A(\Dot{\Tilde{\gamma}}(t), \Tilde{v}(t))$ holds for any $t \in I$ (motivated by Proposition \ref{1-3}). The linearity of this equation ensures that $T_{\Tilde{\gamma}} \mathcal{P}_{tra}(P)$ is a vector space.
\end{definition}

\begin{definition}[Vertical Spaces for $\mathcal{P}_{tra}(P) \times H$]\label{2-13}
As a product, tangent space at any point $(\Tilde{\gamma}, h)$ of $\mathcal{P}_{tra}(P) \times H$ can be seen as $T_{\Tilde{\gamma}} \mathcal{P}_{tra}(P) \times T_hH$. Any vector $(\Tilde{v}, X_H) \in T_{\Tilde{\gamma}} \mathcal{P}_{tra}(P) \times T_hH$ is said to be \textbf{vertical} if $\Tilde{v}(t)$ is vertical in the bundle $P$ for each $t \in I$. Using equation in the above definition, we find that if $\Tilde{v}(t)$ is vertical pointwise, then $\Tilde{v}$ is just a restriction of certain fundamental vector field $\Tilde{X}_G$, induced from some $X_G \in \mathfrak{g}$, to the curve $\Tilde{\gamma}$. 
\end{definition}

Just like the case of principal bundle, in order to separate horizontal vectors from $T_{\Tilde{\gamma}} \mathcal{P}_{tra}(P) \times T_hH$, we need to choose a fake curvature $B$ on $P$.

\begin{definition}[Horizontal Spaces for $\mathcal{P}_{tra}(P) \times H$]\label{2-14}
Suppose $P$ is a principal bundle with a connection $A$ and a fake curvature $B$, any tangent vector $(\Tilde{v}, X_H)$ on $\mathcal{P}_{tra}(P) \times H$ is said to be \textbf{horizontal} with respect to $A$ and $B$ if:
\begin{enumerate}
    \item[(a)] $\Tilde{v}(0)$ is horizontal with respect to $A$.
    \item[(b)] $\int_t B(\Tilde{v},\Dot{\Tilde{\gamma}}) = \omega_H(X_H)$, where $\omega_H$ is the Maurer-Cartan form.
\end{enumerate}
It is straightforward to confirm that these horizontal vectors form a subspace of $T_{\Tilde{\gamma}} \mathcal{P}_{tra}(P) \times T_hH$. 
\end{definition}

\begin{lemma}\label{2-15}
At any point $(\Tilde{\gamma},h) \in \mathcal{P}_{tra}(P) \times H$, vertical and horizontal spaces defined above truly decompose the total tangent space.
\end{lemma}
\begin{proof}
First, suppose vector $(\Tilde{v}, X_H) \in T_{\Tilde{\gamma}} \mathcal{P}_{tra}(P) \times T_hH$ is both vertical and horizontal. Then for any $t \in I$, $\Tilde{v}(t)$ has to be vertical in $P$. At $t = 0$. Condition $(a)$ of horizontal vector also forces $\Tilde{v}(0)$ be zero. Thus $\Tilde{v} \equiv 0$ and $X_H = 0$. 

Now, given an arbitrary vector $(\Tilde{v},X_H)$, let us show how to decompose it: at first, suppose $\Tilde{v}(0) = v_1 + v_2$, where $v_1$ is vertical and $v_2$ is horizontal. For the first component of $(\Tilde{v},X_H)$, let us consider the fundamental vector field $\Tilde{X}_G$ induced by $X_G \in \mathfrak{g}$ such that $\Tilde{X}_G = v_1$ at point $\Tilde{\gamma}(0)$, and let us denote the its restriction to $\Tilde{\gamma}$ by $\Tilde{v}_1$. For the second component, let us define $X_2 = dR_h \big( \int_t B(\Tilde{v},\Dot{\Tilde{\gamma}}) \big) \in T_hH$ (the integral also equals $\int_t B(\Tilde{v}- \Tilde{v}_1,\Dot{\Tilde{\gamma}})$ for $B$ is horizontal). Then by definition, $(\Tilde{v}_1, X_H-X_2)$ is vertical $(\Tilde{v}-\Tilde{v}_1, X_2)$ is horizontal.
\end{proof}

\begin{theorem}[Existence and Uniqueness of Horizontal Lifts for Squares]\label{2-16}
Suppose $P$ is a principal bundle with a connection $A$ and a fake curvature $B$. Let $\Sigma$ be any square. It can be seen as a path connecting each $\Sigma_s \in \mathcal{P}(M)$. (It is also understood as a trajectory or wroldsheet of free string.) Given any point $(\Tilde{\gamma}_0, h_0)$ lifted from $\Sigma_0$ into $\mathcal{P}_{tra}(P) \times H$, there is a unique pair $(\Tilde{\Sigma}, h)$ of smooth maps $\Tilde{\Sigma}(s,t) \in P$ and $h(s) \in H$ such that
\begin{enumerate}
    \item[(a)] $\Tilde{\Sigma}$ is a lift of $\Sigma$ into $P$ and $h(0) = h_0$.
    \item[(a)] For any $s \in I$, $(\partial_s \Tilde{\Sigma}_s, \Dot{h}(s))$ is horizontal with respect to $(A,B)$.
\end{enumerate}
Analogously to the case of paths, this pair $(\Tilde{\Sigma}, h)$ is a higher horizontal lift for square.
\end{theorem}

After so many discussions, it is immediate to see that $\Tilde{\Sigma}$ should be a standard lift and $h(s) = \mathcal{P} \exp{\int_0^s \big(\int_t \Tilde{\Sigma}^\star B (\partial_s, \partial_t) \big) \cdot h_0}$. Still like the case of path, any smooth square $\Sigma$ induces an $(G \times_\alpha H)$-equivariant map $tra_{\Sigma}$ between fibers $\mathcal{P}_{tra}(P) \times H \rvert_{\Sigma_0}$ and $\mathcal{P}_{tra}(P) \times H \rvert_{\Sigma_1}$ by Theorem \ref{2-16}. The map $tra_{\Sigma}$ is called \textbf{parallel transport of strings}.

\bigskip


At the end of Section \ref{1}, we introduce the basic process of categorifying parallel transport along paths. It is natural to speculate that parallel transport on squares can be categorified through 2-category theory. This has been done in \cite{2008arXiv0802.0663S}. Just like the preceding case, higher parallel transport functor can be thought as a collection of 2-functors $F_\alpha$ with thin-path 2-groupoid $\mathcal{P}_2(U_\alpha)$ as source and crossed module $(G,H,\alpha,\tau)$ as target. It encodes the local information of surface holonomy and can be constructed from Definition \ref{2-6}.

Suppose $\Sigma$ is a square contained in $U_\alpha$. squares are assumed to be pinned at endpoints here in consistent with the convention in higher gauge theory. The first step to calculate $tra(\Sigma)$ locally is still taking the lift $s_\alpha \circ \Sigma$. Then we solve the parallel transport $g_s$ of $-dR_{g_s} A_\alpha(\partial_t \Sigma_s) = \Dot{g}_s$ with $g_s(0) = e_H$ for each fixed $s$. The right action of $g_s$ on $s_\alpha \circ \Sigma$ yields a horizontal lift $\Tilde{\Sigma}$ of $\Sigma$. Since $\Sigma$ is pinned at endpoints, the lift is automatically standard. Thus by Definition \ref{2-6}, the surface holonomy $tra(\Sigma)$ with respect to $\Tilde{\Sigma}$ is $\mathcal{P} \exp{\int_s \big( \int_t (\Tilde{\Sigma}^\star B)(\partial_s, \partial_t) \big)}$. When we change fake curvature $B$ by its pullback form $B_\alpha$ for calculation, it is standard to have $tra(\Sigma) = \mathcal{P} \exp{ \int_s \big( \int_t \alpha_{g_s^{-1}} \Sigma^\star B_\alpha (\partial_r, \partial_t) \big) }$ (see Chapter 5 of \cite{2017MATHGAUGE} for more detials), and this formula is essentially the original definition of surface holonomy in \cite{2008arXiv0802.0663S}.

After figuring out how to compute surface holonomy locally, we are in a position to clarify the reason that horizontal composition formula claimed in Proposition \ref{2-8}$(c)$ is different from the classical one: $tra(\Sigma_2 \circ \Sigma_1) =  tra(\Sigma_1) \cdot \alpha_{g^{-1}} tra(\Sigma_2)$ (see \cite{2015JGP....95...28S} and \cite{2018arXiv181110060V} for more details). Suppose squares $\Sigma_1: \gamma_0 \Rightarrow \gamma_1$ and $\Sigma_2:\gamma_0' \Rightarrow \gamma_1'$ are smoothly compatible in the horizontal direction. Following the preceding process in calculating surface holonomy, we denote the standard lifts, obtained from $s_\alpha$ and $A_\alpha$, by $\widetilde{\Sigma_2 \circ \Sigma_1}$, $\Tilde{\Sigma_1}$ and $\Tilde{\Sigma}_2$ with corresponding parallel transports $g_{12,s}$, $g_{1,s}$ and $g_{2,s}$. The subtle difference is due to the choice of standard lifts of $\Sigma_2$. According to the convention in Proposition \ref{2-8}$(c)$, the lift should start at point $\widetilde{\Sigma_2 \circ \Sigma_1}(0,1/2)$, but this is not the case here. After checking carefully, we find that $\widetilde{\Sigma_2 \circ \Sigma_1}(0,1/2) = \Tilde{\Sigma}_2(0,0) \cdot g$, where $g = g_{1,0}(1)$ is the parallel transport of path $\gamma_0$ at its terminal point. Therefore, as suggested by Proposition \ref{2-8}$(a)$, these two formulas for horizontal composition have exactly the same meaning in $U_\alpha$ with respect to their different choices of standard lifts. 


\section{An Attempt to Parallel Transport on Cubes}\label{3}

This last section is devoted to the exploration of parallel transport on cubes. Study on this topic was first proposed by Mackaay and Picken in \cite{2011arXiv0907.2566FP}, where a systematic framework on encoding local 3-dimensional holonomies into Gray 3-functors was established. After that, research on higher 3-gauge theory was carried out in both mathematics and physics. We refer the reader to \cite{2014JMP....55d3506W}, \cite{2014LMaPh.104.1147S} and \cite{2019arXiv190407566R} for more details. In the case, it is necessary to introduce crossed 2-module for it gives rise to the notions of fake 2-curvature and 3-curvature. There is a brief summary on these concepts in the following context before our discussion. 

\subsection{Non-abelian Stokes Theorem in Four Dimension} \label{3.1}

\begin{definition}\label{3-1}
A \textbf{crossed 2-module} consists of a complex of Lie groups: $L \xrightarrow{\delta} H \xrightarrow{\partial} G$, two $G$-actions $\triangleright_H$ on $H$ and $\triangleright_L$ on $L$, and a smooth map $\{-,-\}: H \times H \rightarrow L$, called Peiffer lifting, such that the following conditions hold:
\begin{enumerate}
    \item[(a)] The Lie group morphisms $\delta$ and $\partial$ are equivariant: that is, for any $g \in G$, $h \in H$ and $l \in L$, $\delta(g \triangleright_L l) = g \triangleright_H \delta(l)$ and $\partial(g \triangleright_H h) = Ad_g \partial(l)$.
    
    \item[(b)] The Peiffer lifting $\{-,-\}$ is equivariant: that is, for any $g \in G$ and $h_1,h_2 \in H$, $g \triangleright_L \{h_1, h_2\} = \{g \triangleright_H h_1, g \triangleright_H h_2\}$.
    
    \item[(c)] For any $h_1, h_2 \in H$, $\delta\{h_1, h_2\} = \langle h_1, h_2 \rangle$, where $\langle -, - \rangle: H \times H \rightarrow H$ is defined by $\langle h_1, h_2 \rangle = \big( Ad_{h_1} h_2 \big) \cdot \partial(h_1) \triangleright_H h_2^{-1}$ and called the Peiffer commutator.
    
    \item[(d)] For any $l_1, l_2 \in L$, $[l_1,l_2] = \{\delta(l_1), \delta(l_2) \}$.
    
    \item[(e)] For any $h_1, h_2, h_3 \in H$, $\{h_1 \cdot h_2, \ h_3\} = \{h_1, \ Ad_{h_2}  h_3 \} \cdot \partial(h_1) \triangleright_L \{h_2,h_3\}$ and $\{h_1, \ h_2 \cdot h_3 \} = \{h_1, h_2 \} \cdot \{h_1, h_3\} \cdot \{\langle h_1, h_2 \rangle, \ \partial(h_1) \triangleright_H h_2 \}$.
    
    \item[(f)] For any $h \in H$ and $l \in L$, $\{ \delta(l), h \} \cdot \{ h, \delta(l) \} = l \cdot \big( \partial(h) \triangleright_L l^{-1} \big)$.
\end{enumerate}
Let $\triangleright'$ be an $H$-action on $L$ defined by $h \triangleright' l = l \cdot \{ \delta(l)^{-1}, h \}$. We can prove that $(H,L,\delta,\triangleright')$ forms a crossed module (see \cite{2011arXiv0907.2566FP}). However, this is not the case for $(G,H,\partial,\triangleright_H)$ generally, since the second condition on crossed module (Definition \ref{2-1}$(b)$) will never hold if the Peiffer commutator $\langle -,- \rangle$ is nontrivial.
\end{definition}

Similarly to the case of crossed module, when investigating at the infinitesimal level of crossed 2-module, we obtain the notion of differential crossed 2-module.

\begin{definition}\label{3-2}
Let $\mathfrak{l} \xrightarrow{\delta} \mathfrak{h} \xrightarrow{\partial} \mathfrak{g}$ be a complex of Lie algebras with two $\mathfrak{g}$-actions $\triangleright_H$ on $\mathfrak{h}$ and $\triangleright_L$ on $\mathfrak{l}$ by derivations, and let $\{-,-\}: \mathfrak{h} \times \mathfrak{h} \rightarrow \mathfrak{l}$ be a bilinear map called Peiffer lifting. All of these structures constitute a \textbf{differential crossed 2-module} if the following conditions hold: 
\begin{enumerate}
    \item[(a)] The Lie algebra morphisms $\delta$ and $\partial$ are equivariant: that is, for any $X \in \mathfrak{g}$, $v \in \mathfrak{h}$ and $x \in \mathfrak{l}$, $\delta(X \triangleright_L x) = X \triangleright_H \delta(x)$ and $\partial(X \triangleright_H v) = [X, \partial(v)]$.
    
    \item[(b)] The Peiffer lifting is equivariant: that is, for any $X \in \mathfrak{g}$ and $v_1,v_2,v_3 \in \mathfrak{h}$, $X \triangleright_L \{v_1, v_2\} = \{X \triangleright_H v_1, v_2\} + \{v_1, X \triangleright_H v_2\}$.
    
    \item[(c)] For any $v_1,v_2 \in H$, $\delta\{v_1, v_2\} = \langle v_1, v_2 \rangle$, where $\langle v_1, v_2 \rangle  \vcentcolon =  [v_1, v_2] - \partial(v_1) \triangleright_L v_2$ is the Peiffer commutator of $\mathfrak{h}$.
    
    \item[(d)] For any $x_1,x_2 \in \mathfrak{l}$, $[x_1, x_2] = \{ \delta(x_1),  \delta(x_2) \}$.
    
    \item[(e)] For any $v_1,v_2,v_3 \in \mathfrak{h}$, $\{ [v_1, v_2], v_3 \} = \partial(v_1) \triangleright_L \{v_2, v_3 \} + \{v_1, [v_2, v_3] \} - \partial (v_2) \triangleright_L \{v_1, v_3 \} - \{v_2, [v_1, v_3]\}$ and $\{v_1, [v_2,  v_3] \} = \{\langle v_1, v_2 \rangle, v_3 \} - \{ \langle v_1, v_3 \rangle, v_2 \}$.
    
    \item[(f)] $ \{ \delta(x), v \} + \{v, \delta(x) \} = - \partial(v) \triangleright_L x $, for any $v \in \mathfrak{h}$ and $x \in \mathfrak{l}$.
\end{enumerate}
Let $\triangleright'$ be an $\mathfrak{h}$-action on $\mathfrak{l}$ defined by $v \triangleright' x = - \{\delta(x), v \}$. As one might expect, the collection $(\mathfrak{h}, \mathfrak{l}, \delta, \triangleright')$ forms a differential crossed module, but $(\mathfrak{g}, \mathfrak{h}, \partial, \triangleright_H)$ may not because of the existence of Peiffer commutator.
\end{definition}

\begin{remark}
Generally, Peiffer lifting and Peiffer commutator appearing in differential crossed 2-module are only multilinear maps without further properties. In the present work, we wish to restrict ourselves to a simple case by assuming that all these maps are antisymmetric, just like Lie brackets. In this situation, the first identity in Condition $(e)$ and Condition $(f)$ from Definition \ref{3-2} become:
\begin{enumerate}
    \item[(e')] For any $v_1,v_2,v_3 \in \mathfrak{h}$, $\{ [v_1, v_2], v_3 \} = \{v_1, [v_2, v_3] \} - \{v_2, [v_1, v_3]\}$.
    
    \item[(f')] $\partial(v) \triangleright_L x = 0$, for any $v \in \mathfrak{h}$ and $x \in \mathfrak{l}$.
\end{enumerate}
\end{remark}

Suppose $(G,H,L,\delta,\partial,\triangleright_H,\triangleright_L,\{-,-\})$ is a crossed 2-module and let $P$ be a principal $G$ bundle with a connection $A$. The concepts of fake curvature, surface holonomy and parallel transport of strings still make sense under $(G,H,\partial,\triangleright_H)$ even this sub-structure from the 2-module fails to be a crossed module. However, Proposition \ref{2-8}$(c)$ and Theorem \ref{2-9} are false in this generality. Fortunately, we can remedy at least the Stokes' formula when the Peiffer commutator $\langle -,- \rangle$ is alternating.

\begin{proposition}\label{3-3}
If we define fake curvature and surface holonomy in the setting of crossed 2-module, then the non-abelian Stokes' formula in Theorem \ref{2-9} should be adjusted to:
\begin{align}
    \Big( \mathcal{P} \exp \int_s \int_t & B_r(\partial_s, \partial_t) \Big)^{-1} \cdot \mathcal{P} \exp{\int_s \int_t B_0(\partial_s, \partial_t)}  \notag \\
    & = \mathcal{P} \exp{\int_0^r \int_s Ad_{h_r^{-1}} \Big( \int_t F_B(\partial_r, \partial_s, \partial_t) + \big\langle \int_t B (\partial_r, \partial_t), \int_t B (\partial_s, \partial_t) \big\rangle \Big) }. \notag
\end{align}
It follows immediately that the thin-homotopy invariance of surface holonomy remains valid in this case.
\end{proposition}
\begin{proof}
As in Theorem \ref{2-9}, we need to define reduced forms $a$ and $f$ on $I \times I$. However, different from to the previous calculation, there is one more term appearing in $f(\partial_r, \partial_s)$ due to the Peiffer commutator:
\begin{align}
    \int_t \Big( \big\langle \int_0^t B(\partial_r, \partial_t), B(\partial_s , \partial_t) \big\rangle - \big\langle \int_0^t B(\partial_s, \partial_t), B(\partial_r , \partial_t) \big\rangle \Big). \notag
\end{align}
Thus $f$ fails to satisfy the structure equation of $a$ now, but this problem is easy to solve when the Peiffer commutator is antisymmetric. In this case $\langle a, a \rangle (\partial_r, \partial_s) \vcentcolon = \big\langle \int_t B (\partial_r, \partial_t), \int_t B (\partial_s, \partial_t) \big\rangle$ is a well-defined two form and it is straightforward to see that $F_a = f - \langle a, a \rangle$. Then using the same arguments in Theorem \ref{2-9}, we obtain the adjusted Stokes' formula. 
\end{proof}

The following definitions present the central geometric objects involving in $4$-dimensional non-abelian Stokes' theorem: smooth families of cubes or tesseracts, standard lifts of tesseracts, fake 2-curvature and its 3-curvature.

\begin{definition}[Tesseract]\label{3-4}
A \textbf{tesseract} $T: \Theta_0 \Rightarrow \Theta_1$ is a smooth homotopy between cubes such that the following conditions hold:
\begin{enumerate}
    \item[(a)] $T(q,r,s,i) = \Theta_0(r,s,i) = \Theta_1(r,s,i)$ when $i = 0$ or $1$. 
    \item[(b)] $T(q,r,i,t) = \Theta_0(r,i,t) = \Theta_1(r,i,t)$ when $i = 0$ or $1$.
    \item[(c)] $T(q,i,s,t): \Theta_0(i,s,t) \Rightarrow \Theta_1(i,s,t)$ is a thin-homotopy when $i=0$ or $1$.
\end{enumerate}
It should be noted that, forced by conditions on cubes (see Definition \ref{2-3}), $T(q,r,s,i)$ and $T(q,r,i,t)$ are actually 2-dimensional thin-homotopies when $i=0$ or $1$. Given any tesseract $T$, a lift $\Tilde{T}$ from $T$ is said to be \textbf{standard} if:
\begin{enumerate}
    \item[(a)] For any fixed $q,r,s \in I$, $\Tilde{T}_{q,r,s}(t)$ is horizontal.
    \item[(b)] For any fixed $q,r \in I$, $\Tilde{T}_{q,r}(s,0)$ is horizontal.
    \item[(c)] For any fixed $q \in I$, $\Tilde{T}_q(r,0,0)$ is horizontal.
    \item[(d)] The curve $\Tilde{T}(q,0,0,0)$ is horizontal.
\end{enumerate}
\end{definition}

\begin{definition}\label{3-5}
Suppose $(G,H,L,\delta,\partial,\triangleright_H,\triangleright_L, \{-,-\})$ is a crossed 2-module. Let $P$ be a principal $G$ bundle with a connection $A$ and a fake curvature $B$. A \textbf{fake 2-curvature} $C$ is a $\mathfrak{l}$-valued three form on bundle $P$ satisfying the following conditions:
\begin{enumerate}
    \item[(a)] $C$ is a $\triangleright_L$-equivariant and horizontal form.
    \item[(b)] $\delta(C) = F_B$.  
\end{enumerate}
The \textbf{3-curvature} $F_C$ of $C$ is defined as $dC + A \wedge_{\triangleright_L} C + \{B \wedge B \}$, where $\{B \wedge B \}$ is the wedge produce induced by $\{-,-\}$. Similarly to the case of $2$-curvatures, $dC + A \wedge_{\triangleright_L} C$ is $\triangleright_L$-equivariant and horizontal and so is $F_C$. (See \cite{2011arXiv0907.2566FP} for more details.)
\end{definition}

\begin{theorem}[Non-abelian Stoke' Theorem in Four Dimension]\label{3-6}
Suppose $T: \Theta_0 \Rightarrow \Theta_1$ is a tesseract. With respect to a given standard lift $\Tilde{T}$ of $T$, the $4$-dimensional non-abelian Stoke' formula is written as (notations are explained in the proof):
\begin{align}
    \bigg( \mathcal{P} \exp \int_r \int_s \Big( h_{q,r}(s)^{-1} \triangleright'  \big( \int_t C_q \rvert_{rst} - \{\int_t B_q \rvert_{rt}, \int_t B_q \rvert_{st} \} & \big) \Big) \bigg)^{-1} \notag \\
    \cdot \mathcal{P} \exp \int_r \int_s \Big( h_{0,r}(s)^{-1} \triangleright'  \big( \int_t C_0 \rvert_{rst} -  \{\int_t B_0 \rvert_{rt}, & \int_t B_0 \rvert_{st} \}  \big) \Big)  \notag \\
    = \mathcal{P} \exp \int_0^q \int_r Ad_{(l_q(r))^{-1}} \int_s h_{q,r}(s)^{-1} \triangleright' \Big( \int_t & \big( (F_C)_{qrst} - \{ B \wedge B \}_{qrst} \big) \Big) . \notag
\end{align}
\end{theorem}
\begin{proof}
As one might expect, the argumentation style used here is similar to that in Theorem \ref{2-9}. First, Let us define a $\mathfrak{l}$-valued two form $b$ on $I \times I \times I$  by $b_{q,r,s}(\partial_q, \partial_r) = - \int_t C(\partial_q, \partial_r, \partial_t)$, $b_{q,r,s}(\partial_r, \partial_s) = - \int_t C(\partial_r, \partial_s, \partial_t) $ and $b_{q,r,s}(\partial_q, \partial_s) = - \int_t C(\partial_q, \partial_s, \partial_t)$. In a similar way, we reduce $B$ and $F_B$ to $a$ and $f$ respectively. Although we are working on $I \times I \times I$, it is not hard to show, by Definition \ref{3-4}$(a)$, that the identity $F_a = f - \langle a, a \rangle$ still holds. Correspondence between local gauge fields and connections also enables us to push $a$ and $F_a$ forward to a connection $\Tilde{a}$ and curvature $F_{\Tilde{a}}$ on the trivial bundle $(I \times I \times I) \times H$.

We want to perform Theorem \ref{2-9} on this trivial bundle but the pushforward form $\Tilde{b}$ of $b$ is not a fake curvature with respect to $\Tilde{a}$ for $\delta(b) = f$ which is not equal to $F_a$ now. In order to solve this problem, let us introduce a new form $b' \vcentcolon = b - \{a,a\}$. Condition $(c)$ in Definition \ref{3-2} shows that $\delta(b - \{a,a\}) = f - \langle a, a \rangle = F_a$, and hence the pushforward of $b'$ is a fake curvature with 2-curvature pushed from $F_{b'} \vcentcolon = db - d \{a,a\} + a \wedge_{\triangleright'} b - a \wedge_{\triangleright'} \{a,a\}$. We claim that $d \{a,a\} - a \wedge_{\triangleright'} b + a \wedge_{\triangleright'} \{a,a\}$ vanishes, which means that $F_{b'} = db$. We adopt the notation $C_{qrt}$ instead of $C(\partial_q, \partial_r, \partial_t)$ in the following long computation:
\begin{align}
    (d \{a,a\})_{qrs} = & \{\int_t \partial_q B_{rt}, \int_t B_{st} \} +  \{\int_t B_{rt}, \int_t \partial_q B_{st} \} - \{\int_t \partial_r B_{qt}, \int_t B_{st} \} \notag \\
    & - \{\int_t B_{qt}, \int_t \partial_r B_{st} \} + \{\int_t \partial_s B_{qt}, \int_t B_{rt} \} + \{\int_t B_{qt}, \int_t \partial_s B_{rt} \}. 
    \end{align}    
Definition of $\triangleright'$ indicates that   
    \begin{align}
    -(a \wedge_{\triangleright'} b)_{qrs} = & \big( \int_t B_{qt} \big) \triangleright' \big( \int_t C_{rst} \big) - \big( \int_t B_{rt} \big) \triangleright' \big( \int_t C_{qst} \big) + \big( \int_t B_{st} \big) \triangleright' \big( \int_t C_{qrt} \big) \notag \\
    = & - \{ \int_t (F_B)_{rst}, \int_t B_{qt} \} + \{ \int_t (F_B)_{qst}, \int_t B_{rt} \} - \{ \int_t (F_B)_{qrt}, \int_t B_{st} \}. 
\end{align}
Using the definition of $\triangleright'$ again and noting Conditions $(b)$ and $(c)$ in Definition \ref{3-2}, the last term $(a \wedge_{\triangleright'} \{a,a\})_{qrs}$ equals
\begin{align}
    & - \{ [\int_t B_{rt}, \int_t B_{st} ], \int_t B_{qt} \} + \{ [ \int_t B_{qt}, \int_t B_{st} ], \int_t B_{rt} \} - \{ [ \int_t B_{qt}, \int_t B_{rt} ], \int_t B_{st} \} \notag \\
    & + \{ \int_t \Big( A_r \triangleright_H \big( B_{st} \big) - A_s \triangleright_H \big( B_{rt} \big) \Big),  \int_t B_{qt} \} - \{ \int_t \Big( A_q \triangleright_H \big( B_{st} \big) - A_s \triangleright \big( B_{qt} \big)\Big),  \int_t B_{rt} \} \notag \\
    & + \{ \int_t \Big( A_q \triangleright_H \big( B_{rt} \big) - A_r \triangleright_H 
    \big( B_{st} \big)\Big),  \int_t B_{st} \} 
\end{align}
Terms including Lie brackets in the above computation cancel out by Condition $(e')$ after Definition \ref{3-2}. Using the antisymmetry of $\{-,-\}$ on the first identity, expanding $F_B$ by its definition in the second identity and then adding these three identities together, we see the result is zero and thus confirm the claim.

Let us now consider the identity map $id$ of $I \times I \times I$ with a lift $\Tilde{id}$ on the trivial bundle defined by $\Tilde{id}(q,r,s) = (q,r,s, h_{q,r}(s))$, where $h_{q,r}(s)$ is the path-ordered exponential solved from $\int_t B_{q,r}(\partial_s, \partial_t)$. One may find that the map $id$ fails to be a cube in the sense of Definition \ref{2-3}, but the lift $\Tilde{id}$ has all the ingredients for us to carry out the argumentation in Theorem \ref{2-9}: we need all curves $\Tilde{id}_{q,r}(s)$ to be horizontal with respect to $\Tilde{a}$, but this is straightforward from the constriction of $\Tilde{id}$. For any fixed $q,r \in I$, both $a_{q,r,0}(\partial_r) = - \int_t B_{q,r,0,t}(\partial_r, \partial_t)$ and $a_{q,r,0}(\partial_q) = - \int_t B_{q,r,0,t}(\partial_q, \partial_t)$ vanish by Definition \ref{3-4}$(b)$. Finally, since $T(q,i,s,t)$ is a thin-homotopy for $i = 0$ or $1$, $\int_t C_{q,i,s,t}(\partial_q, \partial_s,\partial_t) - \{\int_t B_{q,i,s,t}(\partial_q,\partial_t), \int_t B_{q,i,s,t}(\partial_s,\partial_t) \} = 0$ and this ensures the validity of Equation \hyperref[(13)]{(13)} in Theorem \ref{2-9}. Let $l_q(r)$ denote the path-ordered exponential $\mathcal{P} \exp{ \int_0^r \int_s \Tilde{b}'_q(\partial_r, \partial_s)}$ of fake curvature $\Tilde{b}'$. Substituting $\Tilde{b}'$, $F_{\Tilde{b}'}$ and $l_q(s)$ into the original Stokes' formula in Theorem \ref{2-9}, we have:
\begin{align}
    \Big( \mathcal{P} \exp \int_r \int_s \big( h_{q,r}(s)^{-1} \triangleright' b'_q & \rvert_{rs} \big)  \Big)^{-1} \cdot  \mathcal{P} \exp{\int_r \int_s \big(  h_{0,r}(s)^{-1} \triangleright' b'_0 \rvert_{rs} \big) } \notag \\
    = & \mathcal{P} \exp{\int_0^q \int_r Ad_{(l_q(r))^{-1}} \int_s h_{q,r}(s)^{-1} \triangleright' (F_{b'})_{qrs} }. 
\end{align}

The last step is finding the relationship between $F_{b'}$ and $F_C$, Let $K_C$ denote the form $dC + A \wedge_{\triangleright} C$. Its reduced form $k$ on $I \times I \times I$ satisfies the following identity:

\begin{align}
    & k_{qrs} = \int_t (dC)_{qrst} + \int_t \big( A_q \triangleright_L C_{rst} \big) - \int_t \big( A_r \triangleright_L C_{qst} \big) + \int_t \big( A_s \triangleright_L C_{qrt} \big) \notag \\
    = & \int_t (dC)_{qrst} + \int_t \big( \partial (\int_0^t B_{qt} ) \triangleright_L C_{rst} \big) -  \int_t \big( \partial (\int_0^t B_{rt} ) \triangleright_L C_{qst} \big)  +  \int_t \big( \partial (\int_0^t B_{st} ) \triangleright_L C_{qrt} \big) \notag \\
    = & \int_t (dC)_{qrst}. 
\end{align}
This is because the antisymmetry of $\{ - , - \}$ forces the action $\partial(v) \triangleright_L x = 0$ for any $v \in \mathfrak{h}$ and $x \in \mathfrak{l}$. Therefore, in the reduced level, $(F_{b'})_{qrs} = k_{qrs} = \int_t \big( (F_C)_{qrst} - \{ B \wedge B \}_{qrst} \big)$. Finally, by making substitutions, we obtain the formula claimed in the beginning of this theorem.
\end{proof}

It would be more concise if we simple write $K_C$ but not $F_C - \{ B \wedge B \}$ in the formula. However, we prefer the current notation since the Peiffer lifting $\{-,-\}$ is shown explicitly on both sides of the Stokes' formula. It can be thought to cancel out commutators occurring in crossed 2-module.

Let us digress to physics for a while. Motivated by previous case of squares, it is easy to thought cubes as trajectories of branes in spacetime. However, for any cube $\Theta$ concerned here, since $\Theta(r,s,i)$ and $\Theta(r,i,t)$ are 2-dimensional thin-homotopies, generally $\Theta$ can only represent vibration but not free movement of branes. The thin-homotopy condition is necessary and cannot be dropped from Thereom \ref{3-6}, so it is expected to see further research relaxing all these restrictions.


\subsection{Volume Holonomy}\label{3.2}

We finish this article by introducing a volume holonomy formula yielded from the previous theorem and discussing some of its basic properties.

\begin{definition}[Volume Holonomy]\label{3-7}
Let $\Theta$ be any cube with a standard lift $\Tilde{\Theta}$. The \textbf{volume holonomy} $tra(\Theta)$ of $\Theta$ with respect to $\Tilde{\Theta}$ is defined as:
\begin{align}
    \mathcal{P} \exp \int_r \int_s \Big( h_r(s)^{-1} \triangleright' \big( \int_t C_{rst} - \{ \int_t B_{rt}, \int_t B_{st} \} \big) \Big), \notag 
\end{align}
where $h_{r}(s) = \mathcal{P} \exp{\int_0^s \int_t B_r(\partial_s, \partial_t)}$. 
\end{definition}

Volume holonomy also has properties similar to those in Proposition \ref{2-8}. Before making the proof, let us investigate some useful facts related to crossed 2-module.
 
\begin{lemma}\label{3-8}
Let us consider the $G$-actions $\triangleright_H: G \times H \rightarrow H$, $\triangleright_L: G \times L \rightarrow L$ and the $H$-action $\triangleright' : H \times L \rightarrow L$. With respect to any fixed $g \in G$ and $h \in H$, the differentials $g \, \triangleright_H$, $g \,\triangleright_L$ and $h \, \triangleright'$ on corresponding Lie algebras have the following properties:
\begin{enumerate}
    \item[(a)] For any $v_1,v_2 \in \mathfrak{h}$, $\{g \triangleright_H v_1, g \triangleright_H v_2 \} = g \triangleright_L \{v_1, v_2\}$ (it is different from the equivariance formula Definition \ref{3-2}$(b)$). 
    \item[(b)] For any $x \in \mathfrak{l}$, $(g \triangleright_H h) \triangleright' (g \triangleright_L x) = g \triangleright_L (h \triangleright' x)$.   
\end{enumerate}
\end{lemma} 
\begin{proof}
\begin{enumerate}
    \item[(a)] For $i=1,2$, let $\gamma_i$ be a smooth path in $H$ starting at $e_H$ with $v_i$ as its tangent vector. Note that $\{-,-\}: \mathfrak{h} \times \mathfrak{h} \rightarrow \mathfrak{l}$ is also a differential from Lie groups, then
    \begin{align}
        \{g \triangleright_H v_1, g \triangleright_H v_2 \} = & \{ \frac{d}{dt} \Big\rvert_0 \Big( g \triangleright_H \gamma_1 \big), \frac{d}{dt} \Big\rvert_0 \big( g \triangleright_H \gamma_2 \big) \} = \frac{d}{dt}\Big\rvert_0 \{ g \triangleright_H \gamma_1, g \triangleright_L \gamma_2 \} \notag \\
        = & \frac{d}{dt}\Big\rvert_0 g \triangleright_L \{ \gamma_1,  \gamma_2 \} = g \triangleright_L \{v_1, v_2\} 
    \end{align}
    
    \item[(b)] Let $\gamma$ be a path in $L$ with $x$ as its initial tangent vector, we have:
    \begin{align}
        & (g \triangleright_H h) \triangleright' (g \triangleright_L x) = (g \triangleright_H h) \triangleright' \frac{d}{dt} \Big\rvert_0 (g \triangleright \gamma) =  \frac{d}{dt} \Big\rvert_0 (g \, \triangleright_H h) \triangleright' (g \triangleright_L \gamma) \notag \\
        = & \frac{d}{dt} \Big\rvert_0 \Big( (g \triangleright_L \gamma) \cdot \{\delta(g \triangleright_L \gamma)^{-1}, g \triangleright_H h \} \Big) = \frac{d}{dt} \Big\rvert_0 \Big( (g \triangleright_L \gamma) \cdot \{g \triangleright_H \delta(\gamma^{-1}), g \triangleright_H h \} \Big) \notag \\
        = & \frac{d}{dt} \Big\rvert_0 \Big( g \triangleright_L \big( \gamma \cdot \{\delta(\gamma^{-1}), h\} \big) \Big ) = \frac{d}{dt} \Big\rvert_0 \big( g \triangleright_L (h \triangleright' \gamma) \big) = g \triangleright_L (h \triangleright' x), 
    \end{align}
    where we use the definition of $\triangleright'$ and equivariance of $\delta$.
\end{enumerate}
\end{proof}

\begin{proposition}\label{3-9}
The following properties hold for volume holonomy (symbol $\circ _r$ and $\circ _s$ are used to denote compositions of cubes in $r$ and $s$ directions):
\begin{enumerate}
    \item[(a)] For any cube $\Theta$, the volume holonomy $tra(\Theta)$ is well-defined up to action $\triangleright_L$: that is, if $\Tilde{\Theta}_1 = \Tilde{\Theta}_2 \cdot g$ are two standard lifts of $\Theta$, then $tra_1(\Theta) = g^{-1} \triangleright_L tra_2(\Theta)$.
    
    \item[(b)] Suppose two cubes $\Theta_1$ and $\Theta_2$ are smoothly compatible in $r$ direction. Let $\Tilde{\Theta}_1$ and $\Tilde{\Theta}_2$ be two standard lifts of these cubes such that $\Tilde{\Theta}_1(1,0,0) = \Tilde{\Theta}_2(0,0,0)$. Then, obviously, $\Tilde{\Theta}_2 \circ _r \Tilde{\Theta}_1$ is a standard lift of $\Theta_2 \circ _r \Theta_1$ and $tra(\Theta_2 \circ _r \Theta_1) = tra(\Theta_2) \cdot tra(\Theta_1)$ with respect to these lifts.
    
    \item[(c)] Suppose two cubes $\Theta_1$ and $\Theta_2$ are smoothly compatible in $s$ direction. Let $\Tilde{\Theta}_1$ and $\Tilde{\Theta}_2$ be two standard lifts such that $\Tilde{\Theta}_1(0,1,0) = \Tilde{\Theta}_2(0,0,0)$. Then, obviously, $\Tilde{\Theta}_2 \circ _s \Tilde{\Theta}_1$ is a standard lift of $\Theta_2 \circ _s \Theta_1$, and with respect to these lifts $tra(\Theta_2 \circ _s \Theta_1) = tra(\Theta_1) \cdot \big( tra(\Sigma)^{-1} \triangleright' tra(\Theta_2) \big)$, where $tra(\Sigma)$ is the surface holonomy on $\Tilde{\Sigma}(s,t) = \Tilde{\Theta}_1(0,s,t)$.
\end{enumerate}
\end{proposition}

\begin{proof}
\begin{enumerate}
    \item[(a)] For $i=1,2$, let us first denote the pullback forms $\Tilde{\Theta}_i^\star B$ and $\Tilde{\Theta}_i^\star C$ by $B^i$ and $C^i$ respectively. The equivariance of $B$ and $C$ and Lemma \ref{3-8}$(a)$ indicate that
    \begin{align}
        & \int_t C^1_{rst} = g^{-1} \triangleright_L \int_t C^2_{rst}, \\ 
        & \{\int_t B^1_{rt}, \int_t B^1_{st}\} = \{ \int_t g^{-1} \triangleright_H B^2)_{rt}, \int_t g^{-1} \triangleright_H B^2_{st}\} = g^{-1} \triangleright_L \{ \int_t B^2_{rt}, \int_t B^2_{st}\}. 
    \end{align}
    Let $h_{r,i}(s)$ denote $\mathcal{P} \exp{\int_s \int_t B^i_r(\partial_s, \partial_t)}$ for $i= 1$ or $2$. Then obviously $h_{r,1}(s) = g^{-1} \triangleright_H h_{r,2}(s)$. Substituting these identities into the volume holonomy formula and using Lemma \ref{3-8}$(b)$, we have:
    \begin{align}
        h_{r,1}(s) \triangleright' \big( \int_t C^1_{rst} - \{ \int_t B^1_{rt}, \int_t B^1_{st} \} \big) = g^{-1} \triangleright_L \Big( h_{r,2}(s) \triangleright' \big( \int_t C^2_{rst} - \{ \int_t B^2_{rt}, B^2_{st} \} \big) \Big). 
    \end{align}
    Thus the identity $tra_1(\Theta) = g^{-1} \triangleright_L tra_2(\Theta)$ holds by Lemma \ref{2-7}$(b)$.
    
    \item[(b)] The second property is still due to the cocycle condition of solutions to $dR_g f = \Dot{g}$.
    
    \item[(c)] In the following computation, we still use the above convention to denote the involved differential forms, for instance, by $C^{12}$ and $B^{12}$. We also abbreviate $\int_t C^{12}_{rst} - \{ \int_t B^{12}_{rt}, \int_t B^{12}_{st} \}$ to $L^{12}(s)$ for simplicity. Let $l_{12}(r)$ denote the path-ordered exponential solved from $f_{12}(r) \vcentcolon = \int_s h_{r,12}(s)^{-1} \triangleright' L^{12}(s)$ with $l_1(r)$ and $l_2(r)$ defined similarly. Using the vertical composition formula for surface holonomy (see Proposition \ref{2-8}$(b)$), we have: 
    \begin{align}
        & \int_s h_{12,r}(s)^{-1} \triangleright' L^{12}(s) = \int_s h_{1,r}(s)^{-1} \triangleright' L^1(s)  + \int_s  h_{1,r}(1)^{-1} \cdot h_{2,r}(s)^{-1} \triangleright' L^2(s). 
    \end{align}
    By Proposition \ref{3-3} and definition of volume holonomy, $h_{1,r}(1)^{-1} = \delta(l_1(r)) \cdot h_{1,r=0}(1)^{-1}$. We also note that, for any crossed 2-module $(G,H,L,\delta,\partial,\triangleright_H,\triangleright_L, \{-,-\})$ the sub-structure $(H,L,\delta,\triangleright')$ is a crossed module. Thus the above identity equals
    \begin{align}
        & \int_s \big( h_{1,r}(s)^{-1} \triangleright' L^1(s) \big) + \int_s \Big( \big( \delta( l_1(r)) \cdot h_{1,r=0}(1)^{-1} \cdot h_{2,r}(s)^{-1} \big) \triangleright' L^2(s) \Big) \notag \\
        & = f_1(r) + Ad_{l_1(r)} \big( tra(\Sigma)^{-1} \triangleright' f_2(r) \big). 
    \end{align}
    Let $l_2'(r)$ denote the path-ordered exponential of $tra(\Sigma)^{-1} \triangleright' f_2(r)$. Following a similar method used in Proposition \ref{2-8}$(c)$, we have: 
    \begin{align}
        & dR_{l_1 \cdot l'_2} f_{12}(r) = dR_{l_1 \cdot l'_2} \Big( f_1(r) + Ad_{l_1(r)} \big( tra(\Sigma)^{-1} \triangleright' f_2(r) \big) \Big) \notag \\
        = & dR_{l'_2} \Dot{l}_1 + dL_{l_1} dR_{l'_2} \big( tra(\Sigma_0)^{-1} \triangleright' f_2(r) \big) = dR_{l'_2} \Dot{l}_1 + dL_{l_1} \Dot{l}'_2 = \frac{d}{dt} (l_1 \cdot l'_2). 
    \end{align}
    Since $dR_{l_{12}} f_{12}(r) = \Dot{l}_{12}$, the uniqueness of solutions indicates that $tra(\Theta_2 \circ _s \Theta_1)  = tra(\Theta_1) \cdot \big( tra(\Sigma)^{-1} \triangleright' tra(\Theta_2) \big)$.
\end{enumerate}
\end{proof}

The situation of composition in $t$ direction is a bit more involved for the corresponding analysis starts with integrals on variable $t$, and we do not discuss it here.

\begin{definition}[Parallel Transport on Cubes]\label{3-10}
Let $\Theta$ be any cube. After all these discussions, the \textbf{parallel transport} on $\Theta$ is defined as the pair $\big( \Tilde{\Theta},tra(\Theta) \big)$ of a standard lift $\Tilde{\Theta}$ and volume holonomy $tra(\Theta)$ evaluated on the lift. One can also consider the equivalence class of $tra(\Theta)$ under $G$-action $\triangleright_L$. Parallel transport on volumes is a gauge invariant.
\end{definition}


\bibliographystyle{alpha}
\bibliography{Bibliography}

\end{document}